\documentclass[11pt, letterpaper]{article}

\usepackage{amsfonts}
\usepackage{amssymb}
\usepackage{amsmath}
\usepackage{amsthm}
\usepackage{latexsym}

\newtheorem{thm}{Theorem}[section]
\newtheorem{lemma}[thm]{Lemma}
\newtheorem{prop}[thm]{Proposition}
\newtheorem{cor}[thm]{Corollary}
\newtheorem{defi}[thm]{Definition}

\newtheorem*{thm*}{Theorem}

\theoremstyle{definition}
\newtheorem*{remark}{Remark}

\newcommand{\lp}{\left(}
\newcommand{\rp}{\right)}
\newcommand{\texto}[1]{\quad\mbox{#1}\quad}
\newcommand{\back}{\backslash}
\newcommand{\R}{\mathbb{R}}

\newcommand{\ep}{\epsilon}
\newcommand{\halfplane}{\R^2_+}

\newcommand{\grad}[2]{|\nabla #1|^{#2}}
\newcommand{\dist}{\text{dist}} 
\newcommand{\HH}{\mathcal H}
\newcommand{\divergence}{\text{div}}

\newcommand{\abs}[1]{\left\lvert #1\right\rvert}
\newcommand{\norm}[1]{\left\lVert #1\right\rVert}

\newcommand{\be}{\begin{equation}}
\newcommand{\ee}{\end{equation}}
\newcommand{\bee}{\begin{equation*}}
\newcommand{\eee}{\end{equation*}}
\newcommand{\bea}{\begin{eqnarray}}
\newcommand{\eea}{\end{eqnarray}}
\newcommand{\bsplit}{\begin{split}}
\newcommand{\esplit}{\end{split}}

\newenvironment{myindentpar}[1]%
{\begin{list}{}%
         {\setlength{\leftmargin}{#1}}%
         \item[]%
}
{\end{list}}

\begin{document}
\title{\bf Gamma convergence of an energy functional related to the fractional Laplacian}
\author{Gonz\'alez, Mar\'ia del Mar\footnote{Universitat Polit\`ecnica de Catalunya, ETSEIB - Departament de Matem\`atica Aplicada I,  Av. Diagonal 647,
08028 Barcelona, SPAIN. Email: mar.gonzalez@upc.edu}}
\date{}
\maketitle

\begin{abstract}
We prove a $\Gamma$-convergence result for an energy functional related to some fractional powers of the Laplacian operator, $(-\Delta)^s$ for $1/2<s<1$, with two singular perturbations, that leads to a two-phase problem. The case $(-\Delta)^{1/2}$ was considered by Alberti-Bouchitt\'e-Seppecher in relation to a model in capillarity with line tension effect. However, the proof in our setting requires some new ingredients such as the Caffarelli-Silvestre extension for the fractional Laplacian and  new trace inequalities for weighted Sobolev spaces.
\end{abstract}

\bigskip




\section{Introduction and statement of the theorem}

\setcounter{equation}{00}

Let $\Omega$ be a bounded domain in $\R^3$ with smooth $\mathcal C^2$ boundary $\partial\Omega$ and let  $h:\Omega\to\R$ be the distance function to the boundary. Fix a real number $-1<a<0$. Let $\alpha,\beta,\alpha',\beta'\in\R$ such that $\alpha<\beta$, $\alpha'<\beta'$ and consider two double-well potentials $W, V:\R\to [0,\infty)$ such that $W$ only vanishes at $\alpha,\beta$, and $V$ only vanishes at $\alpha',\beta'$. For a function $u$ defined in $\Omega$, denote its trace on $\partial\Omega$ by $Tu$. Given $\epsilon>0$,  we study the following energy functional
\be
\label{functional-introduction}F^a_\epsilon[u]:=\ep^{1-a}\int_{\Omega}\grad u 2 h^a +\frac{1}{\ep^{1-a}}\int_{\Omega}W(u)h^{-a}+\lambda_\epsilon\int_{\partial\Omega}V(Tu).\ee
The aim of the present paper is to understand the $\Gamma$-convergence of this functional when $\ep\to 0$ and $\lambda_\ep\to \infty$.

 Note that \eqref{functional-introduction} generalizes the two-phase model  of Alberti-Bouchitt\'e-Seppecher considered in \cite{Alberti-Bouchitte-Seppecher:Phase-transition} in relation to capillarity energy with line tension.  They studied the
 $\Gamma$-convergence of
\be \label{functional-ABS}F_\epsilon[u]:=\ep\int_{\Omega}\grad u 2+\frac{1}{\ep}\int_{\Omega}W(u)+\lambda^0_\epsilon\int_{\partial\Omega}V(Tu)\ee
where $\lambda^0_\ep\to\infty$ is a sequence with some specific behavior as $\ep\to 0$.\\

Historically, this type of models appeared when studying phase transitions. Consider a container $\Omega\subset\R^3$ which is filled with two immiscible and incompressible fluids, or two different phases of the same fluid; equilibrium is achieved when we minimize
\be\label{functional-Modica-limit}E[u]=\sigma\HH^{2}(S_u)\ee
among all the admissible configurations $u\in BV(\Omega, \{\alpha,\beta\})$ with $\int_\Omega u=M_0$.
Here $\HH^2$ is the two-dimensional Hausdorff measure of $S_u$, the singular set of $u$, and $\sigma$ is a constant parameter.
The classical work by Modica \cite{Modica:phase-transitions} established that the following variational model for $u:\Omega\to[\alpha,\beta]$, $\int_\Omega u=M_0$, given by
\be\label{functional-Modica}E_\ep[u]=\epsilon\int_\Omega \grad u 2+\frac{1}{\ep}\int_\Omega W(u),\ee
 $\Gamma$-converges to $E$, for $\sigma:=2\int_\alpha^\beta \sqrt{W(t)}dt$.\\

Here we consider the generalization given by \eqref{functional-introduction}, where the weight $h^a$ is singular at the boundary $\partial\Omega$ since $-1<a<0$. With this modification, $F^a_\ep$ becomes the energy functional related to some fractional powers of the Laplacian, plus two double-well potential terms. Indeed, if  $s=\frac{1-a}{2}$,
then the Euler-Lagrange equation of the functional
\be\label{functional-J}
J[u]=\int_{x\in\mathbb R^n, y\in\mathbb R^+} \abs{\nabla u }^2 y^a \;dxdy
\ee
is just $(-\Delta)^s (Tu)=0$ in $\mathbb R^n$, as it has been shown in the recent work by Caffarelli-Silvestre \cite{Caffarelli-Silvestre}.

On the other hand, the boundary potential term in \eqref{functional-introduction} constitutes a very interesting modification from \eqref{functional-Modica}, and adds new terms in the $\Gamma$-limit. The first result involving boundary integrals was obtained by Modica \cite{Modica:phase-transitions-2}, when $\lambda_\epsilon=1$. Other works can be found in Cabr\'e and Sol\'a-Morales \cite{Cabre-SolaMorales:layer-solutions}, where they look at layer solutions for boundary reactions of the half-Laplacian ($s=1/2$). A refinement of the $\Gamma$-convergence result for $s=1/2$ is being studied by Cabr\'e-C\'onsul \cite{Cabre-Consul:boundary-reactions}. In addition, in the current work \cite{Monneau} by Monneau and the present author, we look at an homogenization problem for a reaction-diffusion equation with half-Laplacian, where we try to understand the interaction energy that appears in the problem.\\

The main theorem in the present paper states that the $\Gamma$-convergence of the sequence $F^a_\ep$, for some suitable scaling $\lambda_\epsilon$, has a similar behavior to the case $s=1/2$ of Alberti-Bouchitt\'e-Seppecher considered in \cite{Alberti-Bouchitte-Seppecher:Phase-transition}. However, the proof needs some new results on the fractional Laplacian: section \ref{section-perturbation} contains a result on singular perturbations of the norm $H^s(\partial\Omega)$, that is indeed the norm of traces of functions in a weighted Sobolev space $W^{1,2}(\Omega,w_a)$ for a suitable weight $w_a$. The second main ingredient, in section \ref{section-traces}, deals with this trace embedding, and gives a more precise control of the Sobolev constant. In particular, in theorem \ref{thm-trace-domain} we understand the relation between the functional \eqref{functional-J} in $\R^2_+$ and the non-local energy of the trace $Tu$ in $H^s(\mathbb R)$ given by
$$\int_{\R\times \R} \frac{\abs{Tu(x')-Tu(x)}^2}{\abs{x'-x}^{1+2s}}\;dx'dx.$$
The proof uses the interesting characterization of the fractional Laplacian in $\R$ as an extension problem to $\R\times \R_+$, given by Caffarelli-Silvestre in \cite{Caffarelli-Silvestre}.\\

Let us fix some notation. Set $-1<a<0$. Let $I:=\{\alpha,\beta\}$, $I':=\{\alpha',\beta'\}$. Assume that there exists $m$ such that $-m\leq \alpha,\alpha',\beta,\beta'\leq m$. Assume that both $W$ and $V$ are positive, increasing in $[m,+\infty)$ and decreasing in $(-\infty,-m]$, with growth at least linear at infinity. As we have mentioned, $W$ only vanishes at the two wells $\alpha<\beta$, and $V$ only at $\alpha'<\beta'$.

Consider the background space $X:=L^1(\Omega)\times L^1(\partial\Omega)$. Fix  $h:\Omega\to [0,+\infty)$ be the distance to the boundary of $\Omega$. Fix the singular weight in $\Omega$ given by $w_a:=h^a$, and consider the weighted Sobolev space $W^{1,2}(\Omega,w_a)$ with norm
$$\norm{u}_{W^{1,2}(\Omega,w_a)}^2:=\int_{\Omega}\abs{u}^2 h^a+\int_{\Omega}\grad u 2 h^a.$$
Note that the trace on $\partial\Omega$ of a function $u\in W^{1,2}(\Omega,w_a)$ is well defined (see theorem \ref{thm-Nekvinda}); denote it by $Tu$.

We define the functional $F^a_\epsilon$ as
\be\label{functional} F^a_\ep[u]:=\left\{
             \begin{split}
               &\epsilon^{1-a}\int_\Omega\grad u 2 h^a+\tfrac{1}{\ep^{1-a}}\int_\Omega W(u)h^{-a}+\lambda_\ep\int_{\partial\Omega} V(Tu), u\in W^{1,2}(\Omega,w_a)\cap X, \\
               &+\infty,  \hbox{ elsewhere in }X.
             \end{split}
           \right.
\ee
On the other hand, given $u\in BV(\Omega,I)$ and $v\in BV(\partial\Omega,I')$, set
\be\label{functional-limit-Phi}\Phi(u,v):=\sigma\HH^2(S_u)+\int_{\partial\Omega} \abs{\mathcal W(Tu)-\mathcal W(v)}+\kappa_s\HH^1(S_v),\ee
and
\be\label{functional-limit}
 F^a(u):=\left\{
             \begin{split}
               & \inf\{\Phi(u,v) : v\in BV(\partial\Omega,I')\},  & \mbox{ if }u\in BV(\Omega,I), \\
               &+\infty,  & \mbox{ elsewhere in }X.
             \end{split}
           \right.
\ee
Here $\HH^i$ denotes the $i$-dimensional Hausdorff measure, that in this case is well defined because of the hypothesis on the bounded variation of $u$, $v$. Also, $S_u$ is the set of all points where $u$ is essentially discontinuous, and the same for $S_v$. The main result of the present paper states that $F^a$ is the $\Gamma$-limit of the functionals $F^a_\epsilon$, for suitable positive constants $\sigma$, $\kappa_s$ and a function $\mathcal W$.

Finally, we set
\be\label{scaling-lambda}
\Lambda_\ep:=\ep^{\frac{1-a}{-a}}.\ee
It will become clear in sections \ref{section-perturbation} and \ref{subsection-boundary-effect} that this is the natural scaling for the problem.\\

We have that, under the above conditions:

\begin{thm}\label{main-theorem} Fix $-1<a<0$, $s=\frac{1-a}{2}$, and assume that \be\label{condition-lambda}\lambda_\epsilon=\frac{1}{\Lambda_\ep}.\ee
Then there exist constants $\sigma,\kappa_s>0$ such that
the functionals $F^a_\ep$ defined in \eqref{functional} $\Gamma$-converge to the functional $F^a$ given in \eqref{functional-limit} and any sequence $(u_\ep)$ with bounded energy is precompact in $X$.
Moreover, $\sigma:=\mathcal W(\beta)-\mathcal W(\alpha)$ where $\mathcal W$ is a primitive of $2\sqrt W$, and $\kappa_s$ is a constant depending only on $s,V$ whose exact value is given in \eqref{kappa}.
\end{thm}

\bigskip

Let us motivate the previous theorem. The potential in the interior, $W$, forces the minimizer $u_\ep$ to take values near the two wells $\alpha,\beta$, while the gradient term in the functional penalizes the jump of the function; thus we create two bulk phases in the interior of the container $\Omega$, namely, $\{u=\alpha\}$, $\{u=\beta\}$, with interphase $S_u$. When $\ep\to 0$, $u_\ep\to u$ and $Tu_\ep \to v$. On the other hand, the second double-well potential $V$ forces the trace to take values near $\alpha',\beta'$, and thus it creates two boundary phases $\{v=\alpha'\}$, $\{v=\beta'\}$, separated by $S_v\subset \partial\Omega$. However, we usually have $Tu\neq v$, and thus additional terms appear in the limit functional \eqref{functional-limit-Phi}. Note that, although many of the arguments would work for a domain $\Omega$ contained in $\mathbb R^n$ for any $n\geq 3$, we restrict ourselves to dimension three so that the energy on $\partial\Omega$ occurring from the boundary phases concentrates over a one-dimensional set $S_v$ of $\partial\Omega$.

We remark here that the proof of the Gamma-convergence result in general follows some well established steps (see Alberti-Bouchitt\'e-Seppecher \cite{Alberti-Bouchitte-Seppecher:Phase-transition} or Palatucci \cite{Palatucci:Gamma-convergence}). However, in our case the shape of the functional requires a deep understanding of the singular factor in the energy and its relation to fractional Laplacian operator - this is precisely the main new idea of the present article.\\

The result is true if either $W$ or $V$ are identically zero. Indeed, if $V$ is zero, we are in the situation of section \ref{section-interior}, while if $W$ is zero, then we can ignore the first two terms in the $\Gamma$-limit \eqref{functional-limit-Phi}.\\

Palatucci (cf. \cite{Palatucci:Gamma-convergence}, or his PhD thesis \cite{Palatucci:thesis}), has considered the $\Gamma$-convergence of the super-quadratic functional
$$P^p_\epsilon[u]:=\ep^{p-2}\int_{\Omega}\grad u p +\frac{1}{\ep^{\frac{p-2}{p-1}}}\int_{\Omega}W(u)+\frac{1}{\ep}\int_{\partial\Omega}V(Tu)$$
when $p>2$. The main difference with ours is the lack of trace inequalities as in theorem \ref{thm-trace-domain} for the case $p>2$.

It has come to our attention that the generalization of the super-quadratic functional $P^p_\ep$ including a singular weight at the boundary $\partial\Omega$ is being completed by Palatucci-Sire \cite{Palatucci-Sire}. It uses some of the results of the present article.\\

Several open questions arise: first, when $0<s<1/2$ we do not know yet how to formulate a Gamma-convergence result due to the lack of layer solutions for the functional \eqref{functional-G}. It may even happen that some other new non-local quantities appear. Also, not much is known in the anisotropic case.\\

The outline of the paper is the following: in section 2 we give some general background on $\Gamma$-convergence and on the fractional Laplacian. In section 3 we study the problem in the interior ignoring the boundary interaction. The following two sections contain the main ingredients: in  section 4 we look at a singular perturbation result for the norm $H^s(E)$ where $E$ is an interval in $\R$, while  the next section deals with some new trace embeddings for weighted Sobolev embeddings. Section 6 contains several technical results that are needed in the main proof, including a dimension reduction argument. Finally, we prove theorem \ref{main-theorem} in the last section.


\section{Some background}

\setcounter{equation}{00}

In order to make this paper self-contained, we present to the reader some standard background. We will denote by $\epsilon\to 0$ any countable sequence converging to zero. The integrals in a domain $\Omega$ are taken with respect the standard Lebesgue measure, while the integrals on $\partial\Omega$ are with respect to the standard Hausdorff measure on the boundary. First, let us give the definition of $\Gamma$-convergence.

\begin{defi}
Let $X$ be a metric space, and for $\epsilon>0$, consider the functional $F_\epsilon:X\to [0,+\infty]$. We say that the sequence $F_\epsilon$ Gamma-converges to $F$ on $X$ as $\epsilon\to 0$ if the following conditions hold:
\begin{itemize}
\item[i.] Lower bound inequality: for every $u\in X$ and every sequence $(u_\epsilon)$ such that $u_\ep \to u$ in $X$, there holds
$$\liminf_{\epsilon\to 0} F_\ep[u_\ep]\geq F[u].$$
\item[ii.] Upper bound inequality: for every $u\in X$ there exists $(u_\ep)$ such that $u_\ep \to u$ in $X$ and
$$\lim_{\epsilon\to 0} F_\ep[u_\ep]= F[u].$$
\end{itemize}
\end{defi}

Together with conditions \emph{i.} and \emph{ii.} a compactness condition is usually proved:
\begin{itemize}
\item[\emph{iii.}] Given a sequence  $(u_\ep)$ when $\ep\to 0$ such that $F_{\ep}[u_\ep]$ is bounded, then $(u_\ep)$ is pre-compact in $X$.
\end{itemize}

 Consider a sequence of functionals $F_\ep$ that $\Gamma$-converges to $F$. If $u_\epsilon$ is a minimizer for $F_\epsilon$, then conditions \emph{i} and \emph{ii.} imply that any limit point of the sequence $(u_\ep)$ is a minimizer for $F$. Condition $\emph{iii.}$ assures that this limit point exists in $X$.\\

The definition and properties of bounded variation functions can be found in \cite{Evans-Gariepy} or \cite{Giusti}. Let $\Omega$ be an open subset of $\R^n$. Note that if $\partial\Omega$ is Lipschitz, the trace of a bounded variation function on $\partial\Omega$ is well defined and it belongs to $L^1(\partial\Omega)$.
Given $f\in BV(\Omega,I)$, we define $S_u$ to be the set of all points where $u$ is essentially discontinuous, that is, it has no approximate limit, and it agrees with the measure theoretic boundary of the set $\{u=\alpha\}$ in $\Omega$.\\

Now we give the relation between weighted Sobolev spaces and their traces.
Let $n>0$, $k\geq 0$ be integers, and $a$, $p$ real numbers, $1<p <\infty$. Let $\Omega$ be a non-empty, open, bounded subset of $\R^n$. Let $M$ be a closed
subset of $\partial\Omega$­ and let $d_M(x)$ be the distance function, $d_M(x) := \dist(x,M)$. For simplicity we shall write $d(x)$ instead of $d_M(x)$. For an integer $m$, $1\leq m \leq n$, we set $Q_m = (0, 1)^m$.

We shall write $(\Omega,M) \in B(k,n)$ for $1 \leq k \leq n - 1$, $n \geq  2$ if
and only if there exists a bi-Lipschitz mapping $B:Q_n  \to \Omega$ such that $B(\bar Q_k) = M$.

By $C^\infty(\bar\Omega)$ we denote the set of real functions $u$ defined on $\bar\Omega$  such that the derivatives $D^\alpha u$ can be continuously extended to $\bar\Omega$ for all multi-indexes $\alpha$. Consider the weight $w=d^a$.
Define the weighted Sobolev space $W^{1,p}(\Omega, w)$ as the closure of $\mathcal C^\infty(\bar\Omega)$ with
respect to the norm
$$\norm{u}^p_{W^{1,p}(\Omega,w)}:=\int_\Omega \abs{u}^p w\; dx+\int_\Omega \abs{Du}^p w\;dx.$$
These norms have been well studied for a certain class of weights, called $A_p$ weights (see \cite{Garcia-Cuerva-Rubio}, \cite{Stein:Harmonic-analysis}). We will also need some other norms for functions $v$ defined on an interval $E\subset\R$ given by
$$\norm{v}_{H^s(E)}^2=\norm{v}_{L^2(E)}^2+\int_{E^2} \frac{\abs{v(x')-v(x)}^2}{\abs{x'-x}^{1+2s}}\;dx'dx$$
for $0<s<1$.

Although it can be stated more generally, we are just interested in the case  $k=n-1$:
\begin{thm}[theorem 2.8 in \cite{Nekvinda:characterization-tracesW}]\label{thm-Nekvinda}
Let $n \geq 2$, $-1 < a < p-1$ and $(\Omega,M) \in B(n -1,n)$. Then there
exists a unique bounded linear operator
\be\label{trace-operator}T: W^{1,p}(\Omega, w) \to  W^{1-\frac{1+a}{p},p}(M)\ee
such that $T u = u|_M$ for all $u \in \mathcal C^\infty(\bar Q_n)$.
\end{thm}

In the special case that $p=2$, $s=\frac{1-a}{2}$, the trace operator \eqref{trace-operator} reads precisely
$$T: W^{1,2}(\Omega, w) \to  H^s(M).$$

\bigskip

The classical reference for the fractional Laplacian $(-\Delta)^s$ is the book by Landkof \cite{Landkof}, although a good reference is the PhD thesis by Luis Silvestre \cite{Silvestre:thesis}. Given $s\in(0,1)$ we define the fractional Laplacian of a function $f:\R^n\to \R$, as a pseudo-differential operator by
$$\widehat{(-\Delta)^s} f(\xi)=\abs{\xi}^{2s}\hat f(\xi),$$
i.e, its principal symbol is $\abs{\xi}^{2s}$. It can also be written as the singular integral
$$(-\Delta)^s f(x)=C_{n,s}\int_{\R^n}\frac{f(x)-f(\xi)-\nabla f(x)\cdot(x-\xi)\chi_{\{\abs{x-\xi}<1\}}}{\abs{x-\xi}^{n+2s}}\;d\xi.$$\\

Caffarelli-Silvestre have developed in \cite{Caffarelli-Silvestre} an equivalent definition using an extension problem, that is crucial in the present work. For a function $f:\R^n\to \R$, we construct the extension $u:\R^n\times [0,+\infty)\to \R$, $u=u(x,y)$, as the solution of the equation
\be\left\{\begin{split}
\Delta_x u+\frac{a}{y}\partial_y u+\partial_{yy} u=& \;0 \quad \mbox{for }x\in\R^n,\; y\in[0,+\infty),\label{equation-extension}\\
u(x,0)=& f(x),
\end{split}\right.\ee
for $s=\frac{1-a}{2}$. This type of degenerate elliptic equations have been studied in \cite{Fabes-Kenig-Serapioni:local-regularity-degenerate}.  Then the fractional Laplacian of $f$ can be recovered as
$$(-\Delta)^s f=c_{n,s}\lim_{y\to 0} y^a \partial_y u,$$
i.e., we are looking at a non-local Dirichlet-to-Neumann operator.
Note that equation \eqref{equation-extension} can be written in divergence form as
$$\divergence(y^a\nabla u)=0,$$
which is the Euler-Lagrange equation for the functional
$$J[u]=\int_{x\in\R^n, y>0} \grad u 2 y^a \,dxdy.$$
To finish, just mention that the Poisson kernel for the fractional Laplacian $(-\Delta)^s$ is given by
\be\label{Poisson-kernel}P(x,y)=c_{n,s}\frac{y^{1-a}}{\lp\abs{x}^2+\abs{y}^2\rp^{\frac{n+1-a}{2}}},\ee
and thus $u=P*_x f$.


\section{$\Gamma$-convergence in the interior}\label{section-interior}

\setcounter{equation}{00}

The theory of phase transitions (see Modica \cite{Modica:phase-transitions}, Modica-Mortola \cite{Modica-Mortola} for the classical references, or Alberti \cite{Alberti:survey}, for a very well written survey), studies the interface between two fluids in a container $\Omega$ neglecting the interaction with the boundary. In particular, it is proven that the functionals $E_\ep$ defined in \eqref{functional-Modica}, $\Gamma$-converge to $E$ given in \eqref{functional-Modica-limit}. In this section we consider the generalization to $E^a_\ep$ given in \eqref{functional-interior}, that involves a singular weight $w_a:=h^a$. However, since the interaction with the boundary is neglected and the weight is regular in the interior, the behavior of $E^a_\ep$ is going to be very similar to $E_\ep$. Indeed, we can easily modify the argument of Modica to prove:

\begin{prop}\label{prop-interior}
Let $A$ be a domain contained in $\Omega$, $h:=\dist(\cdot,\partial\Omega)$, $w_a:=h^a$. Assume, in addition, that
$\dist(A,\partial\Omega)\geq r$ for some $r>0$ fixed. Set
\be\label{functional-interior}
E^a_\ep[u,A]:=\epsilon^{1-a}\int_A\grad u 2 h^a+\frac{1}{\ep^{1-a}}\int_A W(u)h^{-a}\ee
and
\be\label{functional-interior-limit}E^a[u,A]:= \sigma\HH^2(S_u\cap A) \ee
for $\sigma:=2\int_\alpha^\beta \sqrt W$. Then
\begin{itemize}
\item[i.] $\forall u\in BV(A,I)$, $\forall (u_\ep)\subset W^{1,2}(A,w_a)$ such that $u_\ep\to u$ in $L^1(A)$, we have
$$\liminf_{\ep\to 0} E^a_\ep [u_\ep,A]\geq \sigma E^a[u,A].$$
\item[ii.] $\forall u\in BV(A,I)$, there exists $(u_\ep)\subset W^{1,2}(A,w_a)$ such that $u_\ep \to u$ in $L^1(A)$ and
$$\limsup_{\ep\to 0} E^a_\ep [u_\ep,A]\leq \sigma E^a [u,A].$$
Moreover, when $S_u$ is a closed Lipschitz surface in $A$, the functions $u_\epsilon$ may be required to be Lipschitz continuous with constant $\frac{C}{\ep^{1-a} r^a}$,
and to converge uniformly to $u$ in every set with positive distance from $S_u$ and away from the boundary.
\item[iii.] Any sequence $(u_\ep)\subset W^{1,2}(A,w_a)$ with uniformly bounded energies $E_\epsilon^a[u_\ep,A]$ is pre-compact in $X$ and every cluster point belongs to $BV(A,I)$.
\end{itemize}
\end{prop}

\begin{proof}
The proof is essentially the one of Modica and can be found in Alberti \cite{Alberti:survey}.
By a well known truncation argument (\cite{Alberti-Bellettini:anisotropic-model}, lemma 1.14), we can assume that $u:A\to[\alpha,\beta]$. Now, use the inequality $x_1^2+x_2^2\geq 2 x_1 x_2$ with
$x_1=\ep^{(1-a)/2}\grad u {} h^{a/2}$ and $x_2=\ep^{-(1-a)/2} W^{1/2}(u) h^{-a/2} $, then
\be\label{estimate1}E^a_\epsilon[u,A]\geq 2\int_A \sqrt{W(u)}\grad u {}=\int_A \abs{\nabla (\mathcal W(u))}\ee
where $\mathcal W:[\alpha,\beta]\to \R$ is a primitive of $2\sqrt {W}$. This gives \emph{iii.} and \emph{i.} using standard arguments, and can be found exactly in \cite{Alberti:survey}, paragraph 4.5.\\

For \emph{ii.} we need to take care of the weight $w_a$ in the construction. Let $u\in BV(A,I)$. Without loss of generality, we can assume that its singular set $S_u$ is a Lipschitz surface in $A$, even a polyhedral surface of dimension 2 (see \cite{Giusti}, theorem 1.24). We would like to construct a sequence of functions $u_\ep$ that converges to $u$ in $L^1(A)$. First, it is possible to give coordinates $(e(x),\omega)$ in $\Omega$ such that $\omega$ parameterizes $S_u$, and $e(x)\in\R$ is just the signed distance to $S_u$ (positive where $u=\beta$ and negative where $u=\alpha$).
Next, solve the ODE
\be\label{eq1}\theta'=\sqrt{W(\theta)}\ee
with initial condition $\theta (0)=\frac{\alpha+\beta}{2}$. This $\theta$ is the well known optimum profile for the case $a=0$. In order to take into account the weight $w_a$ we set, for every fixed $\omega$, (note that $(0,\omega)$ is any point in $S_u$),
\be\label{eq2}\phi_\omega(t):=\phi(t,\omega)=\theta\lp\frac{t}{h^{a}(0,\omega)}\rp.\ee
Now, for each $\ep>0$, let $t=e(x)/\ep^{1-a}$ and
\be\label{eq3}
u_\ep(x):=\phi_\omega\lp\frac{e(x)}{\ep^{1-a}}\rp.\ee
When $\ep\to 0$ we can compute that
$$\grad {u_\ep} 2(x)= \frac{1}{\epsilon^{2(1-a)}}\left[\phi_\omega'(t)^2+R(\omega,t)o(1)\right].$$
Then we can use the coarea formula to calculate the energy of this function:
\bee \begin{split}
E^a_\ep[u_\ep,A]  = & \;\ep^{1-a}\int_A \grad {u_\ep} 2 h^a +\frac{1}{\ep^{1-a}}\int_A W(u_\ep)h^{-a} \\
=&\frac{1}{\ep^{1-a}}\int_A \left[\lp\phi_\omega'(t)\rp^2h^a \;dx+ W(\phi_\omega(t))h^{-a} + o(1)\right]\;dx \\
=&\int_{-\infty}^{+\infty} \int_{\Sigma_{\ep^{1-a}t}} \left [ \lp\phi_\omega'(t)\rp^2 h^a+ W(\phi_\omega(t))h^{-a}+o(1)\right]\; d\omega \;dt.
\end{split}\eee
When $\ep\to 0$, the level set $\Sigma_{\ep^{1-a}t}$ converges to $S_u\cap A$, and
if $x$ is written in the new coordinates $(e(x),\omega)$, then $h(t,\omega)$ converges to $\dist((0,\omega),\partial \Omega)=h(0,\omega)$. Taking the limit we have that
\bee\limsup_{\ep\to 0} E^a_\ep[u_\ep,A]=\int_{-\infty}^{+\infty}\int_{S_u\cap A} \left [ \lp\phi_\omega'(t)\rp^2 h^a(0,\omega)+ W(\phi_\omega(t))h^{-a}(0,\omega)\right]\;d\omega dt.\eee
But because of \eqref{eq1} and \eqref{eq2}, both terms in the above integration are equal. Then the inequality  $x_1^2+x_2^2\geq 2x_1 x_2$ for two positive numbers $x_1=x_2$ becomes an equality and thus
\bee\begin{split}
\limsup_{\ep\to 0} E^a_\ep[u_\ep,A]&=   \int_{S_u\cap A} \int_{-\infty}^{+\infty} 2 \sqrt{W(\phi_\omega(t))}\phi_{\omega}'(t)\;dtd\omega \\
& =\int_{S_u\cap A} \int_{\alpha}^{\beta} 2\sqrt {W(r)}\; dr d\omega=\sigma \HH^2(S_u\cap A)
\end{split}\eee
as we wished. The Lipschitz constant of $u_\ep$ is computed from \eqref{eq2} and \eqref{eq3}.
\end{proof}


\section{Perturbation of the norm $H^s$}\label{section-perturbation}

\setcounter{equation}{00}

Let $1/2<s<1$. In this section we consider a singular perturbation of the norm $H^s(E)$ when $E$ is a bounded interval in $\R$. As usual, denote  $s=\frac{1-a}{2}$, so $-1<a<0$, and $I':=\left\{\alpha',\beta'\right\}$.
Consider a double-well potential $V:\R\to [0,\infty)$ vanishing only at $I'$, with the same hypothesis as in the introduction.

More precisely, we will study the $\Gamma$-convergence of the functional
\be\label{functional-G}
G^a_\epsilon[v,E]:=\frac{\epsilon^{1-a}}{D_s}\int_{E^2} \frac{\abs{v(x')-v(x)}^2}{\abs{x'-x}^{1+2s}}\;dx'dx +\lambda_\epsilon\int_E V(v)dx,
\ee
when $\ep\to 0$. Although $D_s$ could be any positive constant in this section, we will fix its value as given in \eqref{D_s}.\\

The functional $G^a_\epsilon$ presents the following scaling property that justifies the election of $\lambda_\epsilon$ in \eqref{condition-lambda}. Indeed, if we set $u^\epsilon(x):=u(\Lambda_\epsilon x)$, $E_\epsilon:=\{x: \Lambda_\epsilon x\in E\}$,  then we immediately see that
$$G^a_\epsilon[u,E]=G^a_1[u^\epsilon,E_\epsilon].$$
It is interesting to observe the deterioration of $\lambda_\ep$ when $a\to 0$. In particular, the functional for $s=1/2$, $a=0$ studied in \cite{Alberti-Bouchitte-Seppecher:singular-perturbations} reads
$$G_\epsilon[v,E]:=\frac{\ep}{2\pi}\int_{E^2} \frac{\abs{v(x')-v(x)}^2}{\abs{x'-x}^{2}}\;dx'dx +\lambda^0_\ep\int_E V(v)dx,$$
for $\lim_{\ep\to 0} \ep\log \lambda^0_\ep=K$, $0<K<\infty$.\\

It is natural then consider profiles on the whole real line that minimize the energy, i.e, we consider the following optimal profile problem
\begin{equation}\label{kappa}\kappa_s:=\inf\left\{ \frac{1}{D_s}\int_{\mathbb R^2}\frac{\abs{v(x)-v(x')}^2}{\abs{x-x'}^{1+2s}}\;dxdx'+\int_{\mathbb R} V(v)\;dx\right\},\end{equation}
where the infimum is taken among all the functions in the set
$$\left\{v\in H^s(\mathbb R) \mbox{ : }\lim_{x\to-\infty} v(x)=\alpha',\quad\lim_{x\to+\infty} v(x)=\beta'\right\}.$$
Then we can prove:

\begin{prop}\label{theorem-convergence1}
Let $E$ be an interval in $\R$. The functional defined on $L^1(E)$ given by
\bee\bar G^a_\epsilon [v,E]:=\left\{
                          \begin{array}{ll}
                            G^a_\epsilon[v,E], & \hbox{if} \quad v\in H^s(E),\\
                            +\infty, & \hbox{otherwise,}
                          \end{array}
                        \right.
\eee
$\Gamma$-converges in $L^1(E)$ to
\bee \bar G^a_0 [v,E]:=\left\{
                          \begin{array}{ll}
                            \kappa_s\HH^0(S_v) & \hbox{if} \quad v\in BV(E,I'),\\
                            +\infty, & \hbox{otherwise,}
                          \end{array}
                        \right.
\eee
where  $\HH^0(S_v)$ represents number of points in the singular set of $v\in BV(E,I')$. The constant $\kappa_s>0$ is given by the optimal profile problem \eqref{kappa}. Moreover,
every sequence $(v_\ep)\subset L^1(E)$ with uniformly bounded energies $G^a_\ep[v_\ep,E]$ is precompact in $L^1(E)$ and every cluster point belongs to $BV(E,I')$.
\end{prop}

\bigskip

The proof of this proposition is similar to the work of Garroni-Palatucci \cite{Garroni-Palatucci:singular-perturbation-fractional-norm}, theorem 2.1. Indeed, they considered the functional
\be\label{functional-Palatucci}
\tilde G_{\tilde \epsilon}^p[v,E]:=\tilde \epsilon^{p-2}\int_{E^2}\frac{\abs{v(x)-v(x')}^p}{\abs{x-x'}^p}\;dxdx'+\frac{1}{\tilde \epsilon}\int_{E} V(v)\;dx\ee
for $p>2$. Our case is analogous if we take $p=2s+1$ and $ \tilde\ep=\Lambda_\ep$ since the exponent of the term $\abs{v(x)-v(x')}$ does not play any special role in their proof.

\bigskip

We remark that the case $s=1/2$ was considered in \cite{Alberti-Bouchitte-Seppecher:singular-perturbations}. The main difference with respect to the case $1/2<s<1$, is that here the optimal profile is characterized by the equipartition of the energy between the two terms in the functional. Instead, the logarithmic scaling for $s=1/2$ produces no equipartition of the energy: the limit comes only from the non-local part and does not depend on $V$, i.e., any profile is optimal as far as the transition occurs on a layer of order $\Lambda_\epsilon$. This does not happen in our case; both terms of the energy \eqref{functional-G} are equally important.

A related result on Gamma-convergence and optimal profiles was obtained by Alberti-Bellettini \cite{Alberti-Bellettini:anisotropic-model} for anisotropic singularities that are integrable, unlike our $\frac{1}{\abs{x-x'}^{1+2s}}$.\\

The proof of proposition \ref{theorem-convergence1} is given in the following, and it is divided into three parts: first we show compactness, then we give some preliminary results on the optimal profile problem \eqref{kappa}, and in the last part we show the upper and lower bounds of the $\Gamma$-convergence.


\subsection{Compactness}

We start with a (non optimal) bound from below:

\begin{lemma}\label{lemma-bound-from-below}
Let $\delta$ be given such that $0<\delta<(\beta'-\delta')/2$. For every interval $J\subset E$, $\epsilon>0$, and $(v_\ep)\subset L^1(E)$, let $A_\ep$ and $B_\ep$ be the sets of all points $x\in J$ such that $v_\ep(x)\leq \alpha'+\delta$ and $v_\ep(x)\geq \beta'-\delta$, respectively. Set
\be\label{define-quantities}a_\ep:=\frac{\abs{A_\ep\cap J}}{\abs{J}}, \quad b_\ep:=\frac{\abs{B_\ep\cap J}}{\abs{J}}.\ee
Then
\be\label{inequality-bound-below}
G^a_\ep[v_\ep,J]\geq \ep^{1-a}C_s\frac{(\beta'-\alpha'-2\delta)^2}{\abs{J}^{2s-1}}
\left\{1-\frac{1}{(1-a_\epsilon)^{2s-1}}-\frac{1}{(1-b_\ep)^{2s-1}}\right\}  +C_\delta,\ee
where the constant $C_\delta$, does not depend on $\ep$, and $C_s$ only depends on $s$.
\end{lemma}

\begin{proof}
First note that the non-local part of the energy decreases under monotone rearrangements (see lemma \ref{lemma-rearrangement} below). Set $J=(a_0,b_0)$. Then, if we define $v^*$ to be the non-decreasing rearrangement of $v$, we have that
\bee\begin{split}
\int_{J^2} &\frac{\abs{v(x')-v(x)}^2}{\abs{x'-x}^{1+2s}}\;dx'dx\\
&\geq\int_{J^2} \frac{\abs{v^*(x')-v^*(x)}^2}{\abs{x'-x}^{1+2s}}\;dx'dx\\
&\geq 2(\beta'-\alpha'-2\delta)^2 \int_{a_0}^{a_0+a_\ep\abs{J}}\int_{b_0-b_\ep\abs{J}}^{b_0}\frac{1}{|x-x'|^{1+2s}}dx'dx\\
& = \frac{2(\beta'-\alpha'-2\delta)^2}{2s(2s-1)\abs{J}^{2s-1}} \left[ 1-\frac{1}{(1-a_\ep)^{2s-1}}-\frac{1}{(1-b_\ep)^{2s-1}}+\frac{1}{(1-a_\ep-b_\ep)^{2s-1}}\right].
\end{split}\eee
On the other hand, let $m_\delta:=\inf\left\{V(t) : \alpha'+\delta\leq t\leq \beta'-\delta\right\}$, we obtain
$$\int_J V(v)dx\geq m_\delta \abs{J}\lp 1-a_\ep-b_\ep\rp,$$
so then
$$G^a_\ep[v_\ep,J]\geq \ep^{1-a}\frac{2(\beta'-\alpha'-2\delta)^2}{2s(2s-1)D_s\abs{J}^{2s-1}}+\lambda_\ep m_\delta \abs{J}\lp 1-a_\ep-b_\ep\rp.$$
Minimizing with respect to $\abs{J}(1-a_\ep-b_\ep)$, and taking into account that $s=\frac{1-a}{2}$ we get
\bee\begin{split}
G^a_\ep[v_\ep,J] & \geq \ep^{1-a}\frac{2(\beta'-\alpha'-2\delta)^2}{2s(2s-1)D_s\abs{J}^{2s-1}}\left[ 1-\frac{1}{(1-a_\ep)^{2s-1}}-\frac{1}{(1-b_\ep)^{2s-1}}\right]\\
& + \frac{2^{\frac{1}{2s}} (2s)^{\frac{2s-1}{2s}} (\beta'-\alpha'-2\delta)^{\frac{2}{2s}} m_\delta^{\frac{2s-1}{2s}}}{2s-1}
\end{split}\eee
for every $0<\delta<(\beta'-\alpha')/2$, and the lemma is proved.
\end{proof}

As we have mentioned, the main ingredient in the proof above is the following rearrangement result. In particular, it  tells us that the infimum of the functional must be attained at a non-decreasing function.

\begin{lemma}[\cite{Garsia-Rodemich:rearrangement}]\label{lemma-rearrangement}
Let
$$I_{\Psi,p}(v)=\int_{0}^1\int_0^1 \Psi\lp\frac{v(x)-v(x')}{p(x-x')}\rp\;dxdx'$$
where $\Psi$ and $p$ are restricted as follows:
\begin{itemize}
\item $\Psi(t)$ is defined and continuous on $\mathbb R$ and $\Psi(t)=\Psi(-t)$ is strictly increasing as $\abs{t}\to\infty$.
\item $p(t)$ is defined an continuous on $(-1,1)$ and $p(u)=p(-u)$ is strictly decreasing as $\abs{t}\to 0$.
\item $\Psi(e^x)$ is convex.
\end{itemize}
We also define the non-decreasing rearrangement of $f$ as
$$f^*(x)=\inf\left\{\lambda : m\{t:f(t)\geq\lambda\}\leq x\right\}.$$
Then
$$I_{\Psi,p}(f^*)\leq I_{\Psi,p}(f).$$
\end{lemma}

Now we are ready to prove compactness:

\begin{prop}\label{prop-compactness}
Let $(v_\ep)$ be a sequence in $H^s(E)$ with equibounded energy $G^a_\ep[v_\ep,E]\leq C$. Then $(v_\ep)$ is precompact in $L^1(E)$ and every cluster point belongs to $BV(E,I')$.
\end{prop}

\begin{proof}
Once we have the estimate from lemma \ref{lemma-bound-from-below}, it is a standard argument that we rewrite here for completeness (see theorem 4.4, part (i), in \cite{Alberti-Bouchitte-Seppecher:Phase-transition}), and to show how the estimate \eqref{inequality-bound-below} is used. First, the condition $G^a_\ep[v_\ep,E]\leq C$ implies that
\be\label{term-vanishing}
\int_{E} V(v_\ep)dx\leq C\lambda_\ep^{-1},\ee
and we obtain that $V(v_\ep)\to 0$ in $L^1(E)$ when $\ep \to 0$.
Thanks to the growth assumption on $V$, $(v_\ep)$ is weakly relatively compact in $L^1(E)$, and some subsequence, still denoted by $(v_\ep)$, converges weakly in $L^1(E)$ to some $v$.

In order to prove that this convergence is strong in $L^1(E)$ and that $v\in BV(E,I')$, we need to use the properties of Young measures (see the notes  \cite{Valadier:Young-measures}). Let $\nu_x$ be the Young measure associated with $(v_\ep)$. Since $V$ is a non-negative continuous function in $\mathbb R$, then
$$\int_E \int_{\mathbb R} V(t)d\nu_x(t)\leq \liminf_{\ep\to 0} \int_E V(v_\ep)\; dx.$$
Hence, by \eqref{term-vanishing} we have that
$$\int_{\mathbb R} V(t)d\nu_x(t)=0,\quad \mbox{a.e. }x\in E.$$
Since $V(t)=0$ if and only if $t=\alpha'$ or $t=\beta'$, the probability measure $\nu_x$ is supported on $I'=\{\alpha',\beta'\}$ for a.e. $x$. In other words, there exists a function $\theta:E\to[0,1]$ such that
$$\nu_x(dt)=\theta(x)\delta_{\alpha'}(dt)+(1-\theta)(x)\delta_{\beta'}(dt), \quad x\in E,$$
and
$$v(x)=\theta(x)\alpha'+(1-\theta(x))\beta', \quad x\in E.$$

It remains to prove that $\theta$ belongs to $BV(E,\{0,1\})$. Let us consider the set $S$ of the points where the approximate limits of $\theta$ is neither $0$ nor $1$. For every $N\leq \mathcal H^0(S)$, we can find $N$ disjoint intervals $\{J_j\}_{j=1,\ldots,N}$ such that $J_j\cap S\neq 0$ and such that the quantities $a_\ep^j$ and $b^j_\ep$ defined by \eqref{define-quantities} replacing $J$ by $J_j$ satisfy
$$a_\ep^j \to a_j\in(0,1)\quad\mbox{and}\quad b_\ep^j \to b_j\in(0,1)\quad \mbox{as }\ep\to 0.$$
Then we can apply lemma \ref{lemma-bound-from-below} in the interval $J_j$ and, taking the limit as $\ep \to 0$ in the inequality \eqref{inequality-bound-below} we obtain that
$$\liminf_{\ep\to 0} G^a_\ep[v_\ep,J_j]\geq C_\delta.$$
Finally, we use the sub-additivity of $G^a_\ep$, and we get
$$\liminf_{\ep \to 0} G^a_\ep[v_\ep,E]\geq \sum_{j=1}^N \liminf_{\ep \to 0} G^a_\ep[v_\ep,J_j]\geq N C_\delta.$$
Since $(v_\ep)$ has equi-bounded energy, this implies that $S$ is a finite set. Hence, $\theta\in BV(E,\{0,1\})$ and the proof of the compactness for $G^a_\ep$ is complete.
\end{proof}

\subsection{The optimal profile}\label{subsection-profile}

We remind the reader that we have set the following optimal profile problem
\be\label{optimal-profile-problem}
\kappa_s:=\inf \left\{ G_1^a[v,\mathbb R] :
v\in H^s(\mathbb R),\; \lim_{x\to -\infty} v(x)=\alpha',\; \lim_{x\to +\infty} v(x)=\beta'\right\}.\ee
We would like to show that this infimum is attained.

\begin{lemma}
Let $0<\delta<\frac{1}{2}(\beta'-\alpha')$ and set
$$m_\delta:=\min\{ V(t) : \alpha'+\delta\leq t \leq \beta'-\delta\}.$$
Then, for any $v\in H^s(\R)$ with  $\lim_{x\to -\infty} v(x)=\alpha'$ and $\lim_{x\to +\infty} v(x)=\beta'$, it holds that
$$G^a_1[v,\R]\geq \frac{(2s)^{\frac{2s-1}{2s}} }{2s-1}(\beta'-\alpha'-2\delta)^{\frac{2}{2s}}m_\delta^{\frac{2s-1}{2s}}>0.$$
\end{lemma}

\begin{proof}
It is analogous to proposition 3.1. in \cite{Garroni-Palatucci:singular-perturbation-fractional-norm} for the super-quadratic case, and follows the lines of lemma \ref{lemma-bound-from-below}. Fix $\delta>0$ and fix $v\in H_{loc}^s(\mathbb R)$ such that $\lim_{x\to -\infty} v(x)=\alpha'$, $\lim_{x\to +\infty} v(x)=\beta'$ and $G^a_1[v,\mathbb R]<\infty$. Define
$$I_{\alpha'}:=\{ x\in\mathbb R : v(x)\leq \alpha'+\delta\}\quad \mbox{and}\quad I_{\beta'}:=\{ x\in\mathbb R : v(x)\geq \beta'-\delta\},$$
 and also $J_\delta:=\mathbb R\backslash (I_{\alpha'}\cup I_{\beta'})$. Notice that $I_{\alpha'}$, $I_{\beta'}$ and $J_\delta$ are non-empty, and that $J_\delta$ is bounded, for every fixed $\delta\in (0,(\beta'-\alpha')/2)$. Consider the truncated function
$$v_\delta(x):=(v(x)\vee (\alpha'+\delta))\wedge (\beta'-\delta)\quad\mbox{for every }x\in\mathbb R.$$
It is easy to see that the non local energy decreases under truncation and then it follows that
\be\label{profile1}
\begin{split}
G^a_1[v,\mathbb R]
& \geq \int_{\R\times \R}\frac{\abs{v_\delta(x)-v_\delta(x')}^2}{\abs{x-x'}^{1+2s}}dx dx'+\int_\R V(v)dx \\
& \geq \int_{\R\times \R}\frac{\abs{v_\delta(x)-v_\delta(x')}^2}{\abs{x-x'}^{1+2s}}dx dx'+m_\delta \abs{J_\delta}.
\end{split}\ee
We set
$$x_{\alpha'}:=\min\{ x: v(x)>\alpha'+\delta\},\quad x_{\beta'}:=\max\{ x: v(x)<\beta'-\delta\}$$
Since $v_\delta(x)=\alpha'+\delta$ for every $x<x_{\alpha'}$ and $v_\delta(x)=\beta'-\delta$ for every $x>x_{\beta'}$, for any interval $J\supset [x_{\alpha'},x_{\beta'}]$ the non-decreasing rearrangement $v_\delta^*$ of $v_\delta$ in $J$ does not depend on $J$. Because the rearrangement decreases the non-local energy (lemma \ref{lemma-rearrangement}), we have that
\bee\begin{split}
\int_{\R\times \R}\frac{\abs{v_\delta(x)-v_\delta(x')}^2}{\abs{x-x'}^{1+2s}}dx dx'
&\geq\int_{\R\times R}\frac{\abs{v^*_\delta(x)-v^*_\delta(x')}^2}{\abs{x-x'}^{1+2s}}dx dx' \\
&\geq\int_{-\infty}^{x_{\alpha'}^*} \int_{x_{\beta'}^*}^{+\infty}\frac{\abs{v^*_\delta(x)-v^*_\delta(x')}^2}{\abs{x-x'}^{1+2s}}dx dx',
\end{split}\eee
where $x_{\alpha'}^*:=\sup\{x : v_\delta^*(x)=\alpha'+\delta\}$ and $x_{\beta'}^*:=\inf\{x : v_\delta^*(x)=\beta'-\delta\}$.
So from \eqref{profile1} we deduce that
\bee\begin{split}
G^a_1[v,\mathbb R] &
\geq (\beta'-\alpha'-2\delta)^2 \int_{-\infty}^{x_{\alpha'}^*}\int_{x_{\beta'}^*}^{+\infty}\frac{1}{\abs{x-x'}^{1+2s}}dx dx'+m_\delta\abs{J_\delta}\\
& =\frac{(\beta'-\alpha'-2\delta)^2}{2s(2s-1)\abs{J_\delta}^{2s-1}}+m_\delta\abs{J_\delta}.
\end{split}\eee
Minimizing with respect to $\abs{J_\delta}$ we obtain
$$G^a_1[v,\R]\geq \frac{(2s)^{\frac{2s-1}{2s}} }{2s-1}(\beta'-\alpha'-2\delta)^{\frac{2}{2s}}m_\delta^{\frac{2s-1}{2s}}>0,$$
and the lemma is proved.
\end{proof}

\begin{cor}
The constant $\kappa_s$ is strictly positive.
\end{cor}

We also need an auxiliary optimal profile problem:
for every $T>0$, we minimize
$$\kappa_s^T:=\inf\left\{ G_1^a[v,\mathbb R] : \quad v\in H^s(\mathbb R), \quad v(x)=\alpha'\;\forall x\leq T,\quad v(x)=\beta' \;\forall x\geq T\right\}.$$

\begin{prop}\label{prop-optimal-modified}
The minimum is achieved by a function $\phi^T\in H^s(\R)$ which is non-decreasing and satisfies $\alpha'\leq \phi^T\leq \beta'$. Moreover the sequence $\kappa_s^T$ is non-increasing in $T$ and $\lim_{T \to +\infty} \kappa_s^T=\kappa_s$.
\end{prop}

\begin{proof}
This is essentially proposition 3.2 in \cite{Garroni-Palatucci:singular-perturbation-fractional-norm}.
\end{proof}

Then we have that

\begin{prop}\label{prop-optimum-profile}
The infimum for the optimum profile problem \eqref{optimal-profile-problem} is achieved by a non-decreasing function $\phi$ satisfying $\alpha'\leq \phi\leq \beta'$.
\end{prop}

\begin{proof}
It is a straightforward modification of proposition 3.3 in \cite{Garroni-Palatucci:singular-perturbation-fractional-norm}.
Let $T>0$ and let $\phi^T$
be a non-decreasing minimizer for $\kappa_s^T$. Since the
functions $\phi^T$ are monotone and bounded, by Helly's theorem, there exist a subsequence
$\phi^{T_k}$ of $\phi^T$ and a non-decreasing function $\phi$, bounded by $\alpha'$ and $\beta'$, such
that $\phi^{T_k}$ converges pointwise in $\R$ to $\phi$. By Fatou's lemma and proposition \ref{prop-optimal-modified} we
also have
$$\int_{\R\times\R} \frac{\abs{\phi(x)-\phi(x')}^2}{\abs{x-x'}^{1+2s}}\;dxdx'+\int_\R V(\phi)\;dx\leq \lim_{k\to \infty} \kappa_s^{T_k}=\kappa_s.$$
This $\phi$ is the minimizer we seek.
\end{proof}


\subsection{Lower and upper bound inequalities}

The proof of the upper and lower bound inequalities of proposition \ref{theorem-convergence1} use the optimal profile obtained in proposition \ref{prop-optimum-profile}.
The lower bound inequality is a consequence of the following:

\begin{prop}\label{prop-lower-bound}
Let $J$ be an open interval of $\mathbb R$. Let $(v_\ep)$ be a non-decreasing sequence in $H^s(J)$ and assume that there exist $\bar a,\bar b\in J$, $\bar a<\bar b$, such that for every $\delta>0$, there exists $\ep_\delta$ such that
$$v_\ep(\bar a)\leq \alpha'+\delta\quad\mbox{and}\quad v_\ep(\bar b)\geq \beta'-\delta,\quad \forall \ep\leq \ep_\delta.$$
Then
$$\liminf_{\ep \to 0} G_\ep^a[v_\ep,J]\geq \kappa_s.$$
\end{prop}

\begin{proof}
It is exactly the same as proposition 5.1 in \cite{Garroni-Palatucci:singular-perturbation-fractional-norm} with the modifications indicated after \eqref{functional-Palatucci} and we do not find it necessary to include it here.
\end{proof}

Clearly, an analogue proposition holds in the case of $u_\ep$ non-increasing satisfying the hypothesis with $\bar a>\bar b$. Next, thanks to the compactness result for $G^a_\ep$ from proposition \ref{prop-compactness}, we may assume that the sequence $(v_\ep)$ converges in $L^1(E)$ to some $v\in BV(E,I')$. Hence, the jump set $S_v$ is finite and we can find $N:=\HH^0(S_u)$ disjoint subintervals $\{J_j\}_{j=1,\ldots,N}$  such that $S_v\cap J_j\neq \emptyset$ for every $j=1,\ldots,N$.

Now, let us consider the monotone rearrangement $v^*_{\ep,j}$ of $v_\ep$ in $J_j$. This rearrangement is non-decreasing if $v$ is non-decreasing in $J_j$, and non-increasing otherwise. With this choice, clearly $v_{\ep,j}^*$ converges to $v$ in $L^1(J_j)$ and thus it satisfies the assumptions of proposition \ref{prop-lower-bound} with $J$ replaced by $J_j$. Then, for every $j=1,\ldots,N$, we may conclude that
$$\liminf_{\ep \to 0} G^a_\ep[v_\ep,J_j]\geq \liminf_{\ep \to 0} G_\ep^a[v_{\ep,j}^*, J_j]\geq \kappa_s.$$
Finally, using the subadditivity of $G^a_\ep[v_\ep,\cdot]$, we get
$$\liminf_{\ep \to 0} G^a_\ep[v_\ep,J]\geq \sum_{j=1}^N G^a_\ep[v_\ep,J_j]\geq N\kappa_s=\kappa_s \HH^0(S_v).$$
The lower bound is shown.\\

We prove first the upper bound inequality for a $v$ of the form
\bee
v(x)=\left\{
      \begin{array}{ll}
        \alpha', & \hbox{if }x\leq x_0, \\
        \beta', & \hbox{if } x>x_0.
      \end{array}
    \right.
\eee
Let $T>0$ be fixed and let $\phi^T\in H_{loc}^s(\mathbb R)$ be the minimizer for $\kappa_s^T$ found in  proposition \ref{prop-optimal-modified}. It satisfies $\phi^T(x)=\alpha'$ for all $x\leq -T$, $\phi^T(x)=\beta'$ for all $x\geq T$ and $G^a_1[\phi^T,\R]=\kappa_s^T$.

Let us define, for every $\ep>0$, $v_\ep(x):=\phi^T\lp \frac{x-x_0}{\Lambda_\ep}\rp$, for every $x\in E$. We have that $v_\ep \to v$ in $L^1(E)$ and
$$G_\ep^a[v_\ep,E]=G^a_1[\phi^T, (E-x_0)/ \Lambda_\ep]\leq G_1^a[\phi^T,\R]=\kappa_s^T.$$
By proposition \ref{prop-optimal-modified} again, we get
$$\lim_{T\to +\infty} \limsup_{\ep\to 0} G_\ep^a[v_\ep,E]\leq \kappa_s.$$
Then, by a diagonalization argument, we can construct a sequence $\tilde v_\ep$ converging to $v$ in $L^1(E)$, which satisfies,
$$\limsup_{\ep\to 0} G_\ep^a[\tilde v_\ep, E]\leq \kappa_s.$$
The optimal sequence for an arbitrary $v\in BV(E,I')$ can be easily obtained gluing the sequences constructed above for each single jump of $v$ and taking into account that the long range interactions decay as $\ep \to 0$.


\section{Some new trace inequalities}\label{section-traces}

\setcounter{equation}{00}

It is well known that traces of functions in some Sobolev spaces are represented by functions in Besov spaces (see the book \cite{Adams}, chapters 4 and 7). Here we would like to consider extensions of this result to weighted Sobolev spaces, for the case of domains in $\R^2$.  Let $-1<a<0$, and $\Omega\subset \R^2_+$ a bounded domain with Lipschitz boundary. Theorem \ref{thm-Nekvinda} gives that the class of traces on $M$, a piece of $\partial\Omega$, of the space $W^{1,2}(\Omega,w_a)$ is just the space $H^s(M)$.
However, Nekvinda's proof  does not give an exact value of the constant in the embedding. We seek to give a precise inequality between functions in $W^{1,2}(\Omega,w_a)$ and their traces.

This section is divided into four parts. First we summarize some lemmas that will be needed in the proof. In the second part, we prove the embedding inequality when $\Omega$ is the half-plane $\mathbb R^2_+$ (lemma \ref{lemma-trace-halfplane}). Next, we localize the inequality for a domain $\Omega=(-1,1)\times (0,1)$; this is the main result of the section and it is contained in theorem \ref{thm-trace-domain}. Last, we  show that the inequality is optimal in some sense.


\subsection{Technical lemmas}

In order to understand the energy term related to the fractional Laplacian, first we need to study equation \eqref{equation-extension}. We prove a basic auxiliary result:

\begin{lemma}\label{lemma-Bessel}
Consider the following ODE defined for $y\in\mathbb R^+$:
$$- \varphi(y)+\frac{a}{y}\varphi_y(y)+\varphi_{yy}(y)=0.$$
It has two linearly independent solutions. In particular,
\be\label{solution-ODE}\varphi(y)=y^{s}\left[c_1I_s(y)+c_2K_s(y)\right]\ee
for constants $c_1,c_2\in\R$. Here $I_s,K_s$ are the modified Bessel functions; their asymptotic behavior is given precisely in \eqref{asymptotic1} and \eqref{asymptotic2}.
\end{lemma}

\begin{proof}
We actually have explicit formulas for the solution. Indeed, the change $\varphi(y)=y^s \psi(y)$
gives that $\psi$ must be a solution of the modified Bessel equation
\be\label{Bessel-equation2}
y\psi''+y\psi'-(y^2+s^2)\psi=0.\ee
The books \cite{Watson:Bessel} and \cite{Abramowitz-Stegun} are classical references for this equation. There are two independent solutions of \eqref{Bessel-equation2}: $I_s$, the modified Bessel function of the first kind, and $K_s$, the modified Bessel function of second kind (see 9.6.1. in \cite{Abramowitz-Stegun}). In particular, they have the following asymptotic behavior (cf. 9.6.7 and 9.7.1 in \cite{Abramowitz-Stegun}):
\be\label{asymptotic1}I_s(y)\sim \frac{1}{\Gamma(s+1)}\lp\frac{y}{2}\rp^s,\quad
K_s(y)\sim \frac{\Gamma(s)}{2}\lp\frac{2}{y}\rp^{s}\quad\mbox{when }y\to 0^+.\ee
And when $y\to +\infty$,
\be\label{asymptotic2}I_s(y)\sim \frac{1}{\sqrt{2\pi y}}e^y,\quad
K_s(y)\sim \sqrt{\frac{\pi}{2y}}e^{-y}.\ee
\end{proof}

We show now a result from \cite{Caffarelli-Silvestre} that characterizes the minimizers of the energy in $\R^2_+$:

\begin{lemma}\label{lemma-fourier-characterization}
Given $v\in H^s(\R)$ for $1/2<s<1$, there exists a minimizer $u_0$ of the following variational problem
\be\label{functional-halfspace}J(u)=\int_{\halfplane} \grad u 2 w_a\ee
in the space $Y:=\{u\in L^2_{loc}(\R^2_+,w_a) : \nabla u\in L^2(\R^2_+,w_a)\}$, subject to the constraint $u(x,0)=v(x)$ for all $x\in\R$, and the weight is taken as $w_a:=y^a$.
Moreover,
$$u_0=P *_x v$$
where $P$ is the Poisson kernel defined in \eqref{Poisson-kernel} and
$$\int_{\halfplane}\grad {u_0} 2 w_a\;dxdy=e_s\int_{\R}\abs{\hat v(\xi)}^2 \abs{\xi}^{2s}\;d\xi$$
for some positive constant $e_s$, depending only on $s$, whose precise value is given in \eqref{e_s}.
\end{lemma}

\begin{proof}
The Euler-Lagrange equation of $J$ is
\begin{equation}\left\{\begin{split}\label{equation111}\divergence(y^a\nabla u)& =0 \mbox{ in }\R^2_+ \\
u& =v \mbox{ on }\R\times \{0\}
\end{split}\right.\end{equation}
The main idea is to reduce \eqref{equation111} to an ODE by taking Fourier transform in $x$. We obtain
$$-\abs{\xi}^2\hat u(\xi,y)+\frac{a}{y}\hat u_y(\xi,y)+\hat u_{yy}(\xi,y)=0,$$
that is ODE for each fixed value of $\xi$.

On the other hand, consider the one-dimensional functional for each $\varphi:[0,+\infty)\to \R$ given by
$$\bar J[\varphi]=\int_{\R_+}\left\{ (\varphi')^2(t)+\varphi^2(t)\right\} t^a \,dt$$
subject to the conditions $\varphi(0)=1$ and $\lim_{y\to +\infty} \varphi (t)= 0$. Its Euler-Lagrange equation is given by
$$-\varphi(y)+\frac{a}{y} \varphi_y(y)+ \varphi_{yy}(y)=0,$$
that has been well understood in lemma \ref{lemma-Bessel}. Indeed, if $\varphi_0$ is the minimizer of $\bar J$, its precise formula is given by \eqref{solution-ODE} with $c_1=0$, $c_2=cst(s)$.

Thus we conclude that the minimizer we seek in the present lemma satisfies $\hat u_0(\xi,y)=\hat v(\xi)\varphi_0(\abs{\xi}y)$.
Next, just compute
\begin{equation}\label{eq2000}
\begin{split}
\int_{\R^2_+}\grad {u_0} 2 y^a &=\int_{\R}\int_0^\infty \lp\abs{\xi}^2\abs{\hat u_0}^2+\abs{\partial_y \hat u_0}^2\rp y^a\,dy d\xi \\
& =\int_{\R}\int_0^{\infty}\abs{\hat v(\xi)}^2\abs{\xi}^2 \lp\abs{\varphi_0(\abs{\xi}y)}^2+\abs{\varphi_0'(\abs{\xi}y)}^2\rp y^a\,dyd\xi \\
& = \int_{\R} \abs{\hat v (\xi)}^2 \abs{\xi}^{1-a}\int_0^\infty \lp\abs{\varphi_0(t)}^2+\abs{\varphi_0'(t)}^2\rp t^a \,dtd\xi \\
& =\bar J[\varphi_0]\int_{\R}\abs{\hat v(\xi)}^2 \abs{\xi}^{1-a} d\xi
\end{split}\end{equation}
The lemma is proved, for a constant
\be\label{e_s}e_s:=\bar J[\varphi_0].\ee
\end{proof}

The following result allows to compare the energies in $\R^2_+$ and in $\R\times(0,M)$ for big $M$.

\begin{lemma}\label{lemma-compare-minimizers}
Given $v\in H^s(\R)$, consider the functionals
$$J[u]=\int_{\R^2_+}\grad u 2 w_a \,dxdy \texto{and} J_M[u]=\int_{\R\times (0,M)} \grad u 2 w_a \,dxdy$$
subject to the constraint $Tu=v$ in $\mathbb R$.
Let $u_0$ be the minimizer of $J$ and $u_M$ be the minimizer of $J_M$, in their corresponding weighted function spaces. Given $\ep>0$, there exists $M>0$ depending only on $\epsilon$ and not on $v$ such that
$$\int_{\R\times (0,M)}\grad {u_M} 2 w_a \,dxdy\geq (1-\ep)\int_{\R^2_+}\grad {u_0} 2w_a\,dxdy$$
\end{lemma}

\begin{proof}
We use the ideas of lemma  \ref{lemma-fourier-characterization} in order to reduce the two dimensional problem to an ordinary differential equation through Fourier transform.
The Euler-Lagrange equation of $J_M$ is
\be\left\{\begin{split}\label{equation11}\divergence(y^a\nabla u)& =0 \mbox{ in }\R^2_+, \\
\partial_\nu u &=0 \mbox{ on }\R\times\{M\}, \\
u& =v \mbox{ on }\R\times \{0\},
\end{split}\right.\ee
while the one for $J$ is given by \eqref{equation111}.
Taking Fourier transform in $x$, we immediately realize that both minimizers $u_0$, $u_M$ must satisfy the ODE  (for each value of $\xi$)
$$-\abs{\xi}^2\hat u(\xi,y)+\frac{a}{y}\hat u_y(\xi,y)+\hat u_{yy}(\xi,y)=0.$$
Now consider the two one-dimensional functionals for each $\varphi:[0,+\infty)\to \R$,
$$\bar J[\varphi]=\int_{\R_+}\lp \abs{\varphi'}^2+\varphi^2\rp y^a \,dy \texto{and} \bar J_M[\varphi]=\int_{(0,M)} \lp \abs{\varphi'}^2+\varphi^2\rp y^a \,dy$$
subject to the condition $\varphi(0)=1$.
Let $\varphi_0$ and $\varphi_M$ be the minimizers of $\bar J$ and $\bar J_M$, respectively. Then $\hat u_0(\xi,y)=\hat v(\xi)\varphi_0(\abs{\xi}y)$ and
$\hat u_0(\xi,y)=\hat v(\xi)\varphi_M(\abs{\xi}y)$ because both $\varphi_0$ and $\varphi_M$ satisfy the Euler equation
\be\label{Bessel-equation}- \varphi(y)+\frac{a}{y}\varphi_y(y)+\varphi_{yy}(y)=0.\ee
with boundary conditions
\be\label{boundary-conditions} \begin{split}
\varphi_0(0)=1,&\quad \varphi_0(y)\to 0 \mbox{ when }y\to+\infty,\\
\varphi_M(0)=1,&\quad \varphi'(M)=0.\end{split}\ee
Thus lemma \eqref{lemma-Bessel} gives that
$$\varphi_0(y)=y^{s}\left[c_1I_s(y)+c_2K_s(y)\right],\quad\varphi_M(y)=y^{s}\left[c^M_1I_s(y)+c^M_2K_s(y)\right].$$
If we impose the boundary conditions \eqref{boundary-conditions} at $y=0$, we obtain the value of $c_2=c_2^M=cst(s)$. Now, note that $\partial_y\varphi_M(M)=0$ while $\varphi$ decays at infinity, this fixes $c_1,c^M_1$. In particular, $c_1^M \to c_1=0$ exponentially. We have then that $\varphi_M$ converges to $\varphi_0$ and $\varphi_M'$ to $\varphi_0'$ when $M\to\infty$, and thus, also
\be\label{equation14}\bar J_M[\varphi_M]\to \bar J[\varphi_0].\ee
Then, the computation in \eqref{eq2000} gives that
\be\label{equation12}
\int_{\R^2_+}\grad {u_0} 2 w_a =\bar J[\varphi_0]\int_{\R}\abs{\hat v(\xi)}^2 \abs{\xi}^{1-a} d\xi\ee
while
\be\label{equation13}\int_{\R\times(0,M)}\grad {u_M} 2 w_a =\bar J_M[\varphi_M]\int_{\R}\abs{\hat v(\xi)}^2 \abs{\xi}^{1-a} d\xi.\ee
The lemma is proved by comparing \eqref{equation13} and \eqref{equation12}, because of \eqref{equation14}.
\end{proof}


\subsection{Inequality for a half-plane}

Denote by $\R^2_+$ the half-plane $\R\times(0,+\infty)$, and consider the weight in $\R^2_+$ given by $w_a(x,y):=y^a$ for $x\in\R$, $y>0$.
When $a=0$, Fourier transform methods quickly give that a function $u$ defined on $\R^2_+$ has a well defined trace on $\R\times\{0\}$ and moreover (see \cite{Alberti-Bouchitte-Seppecher:Phase-transition}, lemma 6.2)
\begin{equation}
\label{trace-half-laplacian}\int_{\R^2}\abs{\frac{v(x)-v(x')}{x-x'}}^2\;dxdx'\leq 2\pi \int_{\R^2_+} \grad u 2 \;dxdy.\end{equation}
This constant $2\pi$ is sharp.\\

Trace inequalities for the general case $-1<a<0$ require the characterization of the fractional Laplacian by Caffarelli-Silvestre \cite{Caffarelli-Silvestre} as a Dirichlet-to-Neumann operator. The first result of this subsection deals with the  generalization of \eqref{trace-half-laplacian}.

\begin{lemma}\label{lemma-trace-halfplane}
Let $u$ be a function in $L^2_{loc}(\R^2_+,w_a)$ with derivative in $L^2(\R^2_+,w_a)$. Then the trace of $u$ on $\R\times\{0\}$, call it $v$, is a well defined function $v\in L^2_{loc}(\R)$. Moreover,
\be\label{trace-inequality-halfplane}\int_{\R^2}\frac{\abs{v(x)-v(x')}^2}{\abs{x-x'}^{1+2s}}\;dxdx'\leq D_s \int_{\R^2_+} \grad u 2 w_a\;dxdy,\ee
where $D_s$ is a constant depending only on $s$ but not on $u$, and it is given precisely in \eqref{D_s}.
\end{lemma}

\begin{proof}
First we assume that $u$ is smooth. The result will then follow by approximation. Next, if we apply Plancherel theorem for the Fourier transform we obtain
\be\label{transform1}\bsplit \int_{\R^2}\frac{\abs{v(x)-v(x')}^2}{\abs{x-x'}^{1+2s}}\;dxdx' & =
\int_{\R}\left[\int_{\R}{\abs{v(x+h)-v(x)}^2}\;dx\right]\frac{1}{h^{1+2s}}\;dh \\
&= \int_{\R} \left[\int_{\R}\abs{\hat v(\xi)(e^{2\pi i h \xi}-1)}^2 d\xi\right]\frac{1}{h^{1+2s}}\;dh \\
& =\int_{\R}\left[\int_{\R}\frac{2-2\cos(2\pi h \xi)}{h^{1+2s}}\;dh\right]\abs{\hat v(\xi)}^2 \;d\xi \\
& = d_s (2\pi)^{2s}\int_{\R} \abs{\hat v(\xi)}^2\abs{\xi}^{2s}d\xi,
\end{split}\ee
for
\be\label{d_s}
d_s=\int_{\R}\frac{2-2\cos z}{z^{1+2s}}\;dz.
\ee
On the other hand, let $u_0$ be the minimizer of the functional
$$J(u)=\int_{\halfplane} \grad u 2 w_a$$
found in lemma \ref{lemma-fourier-characterization}. Then
\be\label{transform2}\int_{\halfplane}\grad {u_0} 2 w_a\;dxdy=e_s\int_{\R}\abs{\hat v(\xi)}^2 \abs{\xi}^{2s}\;d\xi.\ee
We set
\be\label{D_s}
D_s:=d_s(2\pi)^{2s}/e_s,
\ee
where $e_s$ is given in  \eqref{e_s} and $d_s$ in \eqref{d_s}.

The proof of  \eqref{trace-inequality-halfplane} is completed from \eqref{transform1}, \eqref{transform2}, and using the fact that $u_0$ is a minimizer of $J$.
\end{proof}

\begin{cor}\label{cor-sharp-halfplane}
Inequality \eqref{trace-inequality-halfplane} is achieved when $u=P*_x v$.
\end{cor}

\begin{proof}
Indeed, the minimizer $u_0$ is given precisely by $u_0=P *_x v$.
\end{proof}

\bigskip

\subsection{Inequality of a bounded domain}

Now we would like to localize inequality \eqref{trace-inequality-halfplane} in a bounded domain $\Omega\subset\halfplane$ such that part of its boundary, say $M\subset \partial\Omega$ lies on $\R\times\{0\}$. Again, we fix the weight $w_a:=y^a$. Nekvinda's work, summarized in theorem \ref{thm-Nekvinda}, assures that the trace of a function $u\in W^{1,2}(\Omega,w_a)$ belongs to the Besov space $H^s(M)$, but it does not give the explicit value of the constant in the embedding.

Here we claim that this constant can be taken as the the constant in $\R^2_+$, the $D_s$ found in the previous subsection:

\begin{thm}\label{thm-trace-domain}
Let $\Omega:=(-1,1)\times (0,1)\subset \R^2_+$ and  $u\in W^{1,2}(\Omega,w_a)$. Then the trace of $u$ on $(-1,1)\times\{0\}$, call it $v$, is a well defined function $v\in H^s(-1,1)$ and, for the same constant $D_s$ as in the lemma \ref{lemma-trace-halfplane}, we have
\be\label{trace-inequality-domain}\int_{(-1,1)^2}\frac{\abs{v(x)-v(x')}^2}{\abs{x-x'}^{1+2s}}\;dxdx'\leq D_s \int_{\Omega} \grad u 2 w_a\;dxdy.\ee
\end{thm}

\begin{proof}
First note that theorem \ref{thm-Nekvinda} quickly gives that $v\in H^s(-1,1)$. Moreover, it is enough to prove inequality \eqref{trace-inequality-domain} for $u\in\mathcal C^\infty(\bar \Omega)$.

Fixed $\epsilon>0$, consider the constant $M>0$ given in lemma \ref{lemma-compare-minimizers}, that is independent of $u,v$. We will show first that
\be\label{inequality-domain-M}\int_{\Omega_M} \grad u 2 w_a\;dxdy \geq D_s^{-1}(1-\epsilon)\int_{(-M,M)^2}\frac{\abs{v(x)-v(x')}^2}{\abs{x-x'}^{1+2s}}\;dxdx',\ee
i.e., that the inequality holds in a domain $\Omega_M:=(-M,M)\times (0,M)$ up to epsilon.
Because the inequality \eqref{inequality-domain-M} is invariant under the rescaling $(\tilde x,\tilde y):=\frac{1}{M}(x,y)$, it also holds for $(-1,1)\times (0,1)$, up to $\ep$. This would finish the proof of \eqref{trace-inequality-domain}.\\

We compare the energy of $u$ to the energy of $u_1$,  the minimizer of the functional
$$J_{\Omega_M}[u]=\int_{\Omega_M} \grad u 2 w_a \,dxdy,$$
subject to the restriction $Tu_1=v$. We now try to extend $u_1$ defined on $\Omega_M$ to the whole $\mathbb R^2_+$, in order to use lemma \ref{lemma-trace-halfplane}. We do it in several steps.

First, because $u_1$ satisfies a zero Neumann condition on the boundary $\{-M,M\}\times (0,M)$, it is possible to reflect it evenly $N$ times  to obtain a solution on $\Omega_{NM}=A_{NM}\times (0,M)$ where $A_{NM}$ is an interval of length $NM$. We still denote that extension by $u_1$. Call $v_1$ to be the trace of $u_1$ on $y=0$.
Now we would like to extend $u_1$ to the whole $\R\times (0,M)$. For this, first extend $v_1$ continuously to $\R$ such that the extension, call it $v_2$,  belongs to $H^s(\R)$, and solve the Caffarelli-Silvestre extension problem
\bee\left\{\begin{split}\divergence(y^a\nabla u_2)& =0 \mbox{ in }\R^2_+ \\
u_2& =v_2 \mbox{ on }\R\times \{0\}.
\end{split}\right.\eee
Let $B:=[b_1-1,b_1]\cup[b_2,b_2+1]$ where $b_1=\inf (A_{NM})$, $b_2=\sup (A_{NM})$, and let $\eta:\R\to \R^+$ be a smooth cutoff such that $\eta=1$ in $A_{NM}$ and $\eta=0$ outside $A_{NM}\cup B$. Consider the function obtained by the gluing
$$u_3:=\eta u_1+(1-\eta)u_2,$$
it is defined on $\R\times (0,M)$.

Next, let $u_M:\R\times(0,M)\to \R$ be the minimizer of the functional
$$J_M[u]=\int_{\R\times(0,M)}\grad u 2 w_a \,dxdy$$
and $u_0:\R^2_+\to\R$ be the minimizer of
$$J[u]=\int_{\R^2_+}\grad u 2 w_a \,dxdy,$$
both subject to the condition $T u=v_2$ on $y=0$. In particular, we can use the Poisson kernel \eqref{Poisson-kernel} to write $u_0=P *_x v_2$. Lemma \ref{lemma-compare-minimizers} implies that
\be\label{equation17}\int_{\R\times(0,M)} \grad {u_3} 2 w_a \geq \int_{\R\times(0,M)} \grad {u_M} 2 w_a\geq (1-\ep)\int_{\R^2_+}\grad {u_0} 2 w_a.\ee
Applying corollary \ref{cor-sharp-halfplane} to $u_0=P*_x v_2$, we obtain that
\be\label{equation15}\begin{split}
\int_{\R^2_+}\grad {u_0} 2 w_a & =D_s^{-1}\int_{\R\times\R} \frac{\abs{v_2(x)-v_2(x')}^2}{\abs{x-x'}^{1+2s}}\,dxdx' \\
& \geq ND_s^{-1}\int_{(-M,M)^2} \frac{\abs{v(x)-v(x')}^2}{\abs{x-x'}^{1+2s}}\,dxdx',\end{split}\ee
where the last inequality holds just by dropping part of the domain of integration.
On the other hand, because of our gluing construction,
\begin{equation}
\label{equation16}\begin{split}
\int_{\R\times(0,M)}\grad {u_3} 2 w_a & = N\int_{\Omega_M} \grad {u_1} 2 w_a+\int_{B\times (0,M)}\grad {u_3}2 w_a+\int_{[\R\backslash(A_{NM}\cup B)]\times (0,M)} \grad {u_2}2 w_a \\
& =: N\int_{\Omega_M} \grad {u_1} 2 w_a +I_3+I_2.\end{split}\ee

We would like to show that the last two terms $I_2, I_3$ above are bounded independently of $N$. First of all, we can use the Poisson kernel \eqref{Poisson-kernel} to write explicit formulas for $u_2=P *_x v_2$, i.e,
$$u_2(x,y)=c_{2,s}\int_{\R}  \frac{y^{1-a} v_2(\xi)}  {(\abs{x-\xi}^2+\abs{y}^2) ^{\frac{2-a}{2}}}d\xi=
c_{2,s}\int_{\mathbb R} \frac{v_2(x-zy)}{\lp z^2+1\rp^{\frac{2-a}{2}}}\;dz$$
after the change $z=\frac{x-\xi}{y}$.
Because the original $v$ was differentiable, we have that $v_2$ has bounded derivative (it is smooth except perhaps at the reflection points). Then
$$\partial_x u_2(x,y)=C\int_{\mathbb R} \frac{v_2'(x-zy)}{\lp z^2+1\rp^{\frac{2-a}{2}}}\;dz$$
and
$$\partial_y u_2(x,y)=C \int_{\mathbb R} \frac{(-z)v_2'(x-zy)}{\lp z^2+1\rp^{\frac{2-a}{2}}}\;dz.$$
In the following, $C$ will be a positive constant that may change from line to line. We consider the following auxiliary term $I$ and can compute, using H\"older, that
\begin{equation}
\begin{split}
I: & =\int_{x\in(b_2,+\infty),y\in(0,M)} \grad {u_2} 2 y^a\;dxdy \\
& \leq C\int_{x\in(b_2,+\infty),y\in(0,M)} \int_{z\in\mathbb R} y^a\frac{\left[v_2'(x-zy)\right]^2\max{\{1,\abs z ^2\}}}{\lp z^2+1\rp^{2-a}}\;dzdxdy.\end{split}\end{equation}
We could have extended $v_2$ to constant (or very decaying to constant) on the interval $[b_2,+\infty)$. Thus it is enough to consider the change $\theta=x-zy-b_2$ so that
\be\begin{split}
 I& \leq C\int_{z\in \mathbb R} \frac{\max{\{1,\abs z ^2\}}}{\lp z^2+1\rp^{2-a}}\int_{y\in(0,M)} y^a\int_{\theta\in (-zy,0)} v_2'(\theta+b_2)^2\;d\theta dydz\\
& \leq C\int_{z\in \mathbb R} \frac{\abs{z}\max{\{1,\abs z ^2\}}}{\lp z^2+1\rp^{2-a}}\;dx
\int_{y\in(0,M)} y^{a+1}dy. \end{split}\ee
This last integral is bounded independently of $N$ because $-1<a<0$. Similar arguments give that the terms $I_2$, $I_3$ from \eqref{equation16} are $o(N)$.\\

Then, combining \eqref{equation17}, \eqref{equation15} and \eqref{equation16} we obtain that
$$N\int_{\Omega_M}\grad {u_1} 2 w_a +o(N)\geq (1-\ep)D_s^{-1} N\int_{(-M,M)^2}\frac{\abs{v(x)-v(x')}^2}{\abs{x-x'}^{1+2s}}\,dxdx'.$$
Divide the previous inequality by $N$ and let $N\to\infty$, so that
\begin{equation}\label{eq40}
\int_{\Omega_M}\grad {u_1} 2 w_a \geq (1-\ep)D_s^{-1} \int_{(-M,M)^2}\frac{\abs{v(x)-v(x')}^2}{\abs{x-x'}^{1+2s}}\,dxdx'.\end{equation}
To finish just note that
\begin{equation}\label{eq41}
\int_{\Omega_M} \grad u 2 w_a\,dxdy \geq\int_{\Omega_M}\grad {u_1} 2w_a\,dxdy\end{equation}
because $u_1$ was a minimizer of $J_{\Omega_M}$. The proposition follows from \eqref{eq40} and \eqref{eq41}, after rescaling back.
\end{proof}


\subsection{Optimality}

In the remaining of the section we would like to show that the constant $D_s$ in \eqref{trace-inequality-domain} is optimal in some sense.

\begin{prop}\label{prop-sharp}
For each $\ep>0$, there exists a function $u_\ep$ defined on $(-1,1)\times(0,1)$ such that
\be\label{eq10}
\ep^{1-a}\int_{(-1,1)\times(0,1)} \grad {u_\ep} 2 w_a =D_s^{-1}\ep^{1-a}\int_{(-1,1)^2}\frac{\abs{T u_\ep(x)-T u_\ep{(x')}}^2}{\abs{x-x'}^{1+2s}}dxdx'+R_\ep
\ee
with $\lim_{\ep\to 0 }|R_\ep|=0$, and this term is of lower order when $\ep\to 0$.
\end{prop}

\begin{proof}
Fix a domain $\Omega=(-1,1)\times(0,1)$, and for each $\epsilon>0$ consider the scaling $\Lambda_\ep>0$ given by \eqref{scaling-lambda}.
Set $v_\ep:(-1,1)\to \R$ to be the function that satisfies $v_\ep(x)=0$ if $x\in(-1,-\Lambda_\ep/2]$, $v_\ep(x)=1$ if $x\in[\Lambda_\ep/2,1)$, and linear in the interval $[-\Lambda_\ep/2,\Lambda_\ep/2]$. Extend it to $\R$, by making it constant on $(-\infty,-1)$ and $(1,+\infty)$, and denote this extension by $\tilde v:=\tilde v_\ep$.
Now construct a function $u:=u_\ep$ defined on $\R^2_+$ with trace $\tilde v$ as  $u=P*_x\tilde v$.
Corollary \ref{cor-sharp-halfplane} tells that for $u$ constructed this way,
\be\label{eq5}
\int_{\mathbb R^2_+} \grad u 2 w_a \,dxdy=D_s^{-1}\int_{\R\times\R}\frac{\abs{\tilde v(x)- \tilde v(x')}^2}{\abs{x-x'}^{1+2s}}\,dxdx'.
\ee
We try now to restrict the domain of integration to $\Omega$ by estimating the remaining terms.
Let $\Omega_\ep:=((-1,-\Lambda_\ep/2]\cup[\Lambda_\ep/2,1))\times (\Lambda_\ep/2,1)$.
A straightforward computation from the Poisson formula
\bee u(x,y)=C\int_{\R} \frac{y^{1-a} \tilde v(\xi)}{\lp\abs{x-\xi}^2+\abs{y}^2\rp^{(2-a)/2}} \,d\xi\eee
gives that
$$\int_{\Omega_\ep} \grad u 2 w_a \,dxdy=\frac{C}{\Lambda_\ep^{2s-1}}+o\lp\frac{1}{\Lambda_\ep^{2s-1}}\rp,$$
while
$$\int_{\R^2_+} \grad u 2 w_a \,dxdy-\int_{\Omega_\ep} \grad u 2 w_a \,dxdy=o\lp\frac{1}{\Lambda_\ep^{2s-1}}\rp.$$

On the other hand, let $A_\ep=(-1,\Lambda_\ep/2]$, $B_\ep=[\Lambda_\ep/2,1)$. We also have
$$\int_{A_\ep\times B_\ep}\frac{\abs{\tilde v(x)-\tilde v(x')}^2}{\abs{x-x'}^{1+2s}}\,dxdx'=\frac{C'}{\Lambda_\ep^{2s-1}}+o\lp\frac{1}{\Lambda_\ep^{2s-1}}\rp$$
while
$$\int_{\R^2\backslash((A_\ep\times B_\ep)\cup (B_\ep\times A_\ep))}\frac{\abs{\tilde v(x)-\tilde v(x')}^2}{\abs{x-x'}^{1+2s}}\,dxdx'=o\lp\frac{1}{\Lambda_\ep^{2s-1}}\rp$$
Because of equality \eqref{eq5}, the previous estimates, and the fact that
$$\ep^{1-a}\frac{1}{\Lambda_\ep^{2s-1}}=1$$
we conclude that the $u_\ep$, $v_\ep$ we have constructed satisfy \eqref{eq10}.
\end{proof}

\begin{remark}
We have constructed a piecewise linear function. However, the above proposition is still true as long as the transition occurs on a layer of length $\Lambda_\ep$.
\end{remark}


\section{Preliminary results}\label{section-preliminary}

\setcounter{equation}{00}

We are going to prove the main theorem \ref{main-theorem} by localizing the functional into several regions. Section \ref{section-interior} takes care of the behavior of the functional in the interior of the domain $\Omega$. Thus, it remains to study the interaction the interaction with the wall of the container $\partial\Omega$.
In the present section we give some preliminary results in that regard.

First, we look at a small neighborhood of $x\in\partial\Omega$, and reduce the problem to the study of a small  neighborhood $B_r\times (0,r)\subset \mathbb R^3_+$. Then we reduce the dimension from three to two through a slicing argument, so that it is enough to consider a subdomain $\Omega=(-r,r)\times (0,r)\subset \R^2_+$,  and thus, the results of sections \ref{section-perturbation} and \ref{section-traces} can be applied.

Let us remind the reader that the Gamma-limit is going to be expressed in terms of the functional \eqref{functional-limit-Phi}, whose exact expression is
$$\Phi(u,v):=\sigma\HH^2(S_u)+\int_{\partial\Omega} \abs{\mathcal W(Tu)-\mathcal W(v)}+\kappa_s\HH^1(S_v).$$
The first term in $\Phi$ comes from the behavior in the interior (section \ref{section-interior}). In proposition \ref{prop-wall} we will consider the ``wall effect", that explains the presence of the second term; while the remaining of the section is devoted to the ``boundary effect", that deals with the third term.\\

We define the localization of  the functional $F^a_\ep$ as follows. For every open set $A\subset \Omega$ and any $A'\subset \partial A$, we set
\be\label{localization-functional}
F^a_\ep[u,A,A']:=\ep^{1-a}\int_{A}\grad u 2 w_a+\frac{1}{\ep^{1-a}}\int_{A}W(u)w_a^{-1}+\lambda_\ep\int_{A'}V(Tu).
\ee
Note that we easily recover the original functional as $F^a_\epsilon[u]=F_\epsilon^a[u,\Omega,\partial\Omega]$.

\subsection{Reduction to the flat case}

Fixed $x\in\partial\Omega$, we will consider small ``cubical" neighborhoods $Q_r(x)$ near the boundary, of size $r$, such that $\Omega\cap Q_r(x)$ is equivalent to the cylinder $D_r:=B_r\times (0,r)$. Here $B_r$ is the two-dimensional ball of radius $r$ centered an the origin. Let $E_r$ be the boundary part given by $E_r:=B_r\times \{0\}$.
In order to evaluate the error in the deformation, we need to introduce the notion of isometry defect.

Given two domains $A_1,A_2\subset \mathbb R^3$ and a bi-Lipschitz homeomorphism $\Psi:\bar A_1 \to \bar A_2$, the isometry defect $\delta(\Psi)$ of $\Psi$ is the smallest constant $\delta$ such that
\be\label{defi-isometry-defect}\dist(D\Psi(x),O(3))\leq \delta \quad\mbox{for a.e. }x\in A_1.\ee
Here $O(3)$ is the set of linear isometries on $\mathbb R^3$, and $D\Psi(x)$ is regarded as a linear mapping of $\mathbb R^3$ into $\mathbb R^3$. Let $I$ be the identity map on $\mathbb R^3$. The distance between linear mappings is induced by the norm $\norm{\cdot}$, which for every linear map $T$ is defined as
$$\norm{T}=\sup_{\abs{v}\leq 1} \abs{Tv}.$$

The following proposition shows that the localized energy $F^a_\ep[u,\Omega\cap Q_r(x),\partial\Omega\cap Q_r(x)]$ can be replaced by the energy $F^a_\ep [u, D_r, E_r]$.

\begin{prop}\label{prop-flatten}
Let $\Omega$ be a domain in $\mathbb R^3$ with $\mathcal C^2$ boundary $\partial\Omega$ in $\mathcal C^2$. Then for every $x\in\partial\Omega$ and every positive $r$ smaller than a certain critical value $r_x>0$, there exists a diffeomorphism $\Psi_r:\bar D_r \to \overline{\Omega\cap Q_r(x)}$ such that
\begin{enumerate}
\item $\Psi_r$ takes $D_r$ onto $\Omega\cap Q_r(x)$ and $E_r$ onto $\partial\Omega \cap Q_r(x)$.
\item $\Psi_r$ is of class $\mathcal C^1$ on $D_r$  and $\norm{D\Phi_r -Id}\leq \delta_r$ everywhere on $D_r$, where $\delta_r \to 0$ as $r\to 0$.
\end{enumerate}
In particular, the isometry defect of $\Psi_r$ vanishes as $r\to 0$.
Moreover,
$$F^a_\epsilon[u,Q_r(x)\cap\Omega,Q_r(x)\cap\partial\Omega]
\geq \lp 1-\delta(\Psi_r)\rp^{5} F_\ep^a[u \circ \Psi_r,D_r,E_r].$$
\end{prop}

\begin{proof}
It is essentially contained in propositions 4.9 and 4.10 from \cite{Alberti-Bouchitte-Seppecher:Phase-transition}, although in our case we need to make some modifications due to the presence of a weight.

Because of the smoothness assumptions on $\partial\Omega$, we can parameterize a small neighborhood of $x\in\partial\Omega$ with coordinates $(t,\rho)$, where $\rho=\dist(\cdot,\partial\Omega)$,  $t$ is the coordinate parameterizing each level set of $\rho$ and such that $\rho\in(0,r)$, and $t\in B_r$. The change of coordinates map, call it $\Psi_r$ is a diffeomorphism with $D\Psi_r (x)=I$. It is clear than \be\label{isometry-defect}\delta:=\delta(\Psi_r)\to 0,\quad \mbox{when } r\to 0.\ee
Moreover,
\be\label{change-variable-grad}\abs{D(u\circ \Psi_r)}\leq (1+\delta)\abs{(Du)\circ \Psi_r}.\ee
Let $J$ be the Jacobian determinant of $\Psi_r$ on $D_r$ and $J'$ be the one of $\Psi_r|_{E_r}$. They satisfy $\abs{J}, \abs{J'}\leq (1+\delta)^{3}$ a.e.. Using the change of variable formula and \eqref{change-variable-grad} we see that
\bee\begin{split}
F^a_\ep  [u \circ \Psi_r,D_r,E_r]&=
 \ep^{1-a}\int_{D_r}\grad {(u\circ\Psi_r)} 2 \rho^a J d\rho dt\\
& +\frac{1}{\ep^{1-a}}\int_{D_r}W(u\circ \Psi_r)\rho^{-a}Jd\rho dt +\lambda_\ep\int_{E_r}V(T(u\circ \Psi_r))J'dt \\
& \leq  \lp 1+\delta(\Psi_r)\rp^{5} F^a_\epsilon [u,Q_r(x)\cap\Omega,Q_r(x)\cap\partial\Omega], \\
\end{split}\eee
as we wished, because $\frac{1}{1+\delta}\geq 1-\delta$.
\end{proof}

\begin{remark}
The regularity of $\partial\Omega$ could be reduced to $\mathcal C^{1,\alpha}$ and still have
\eqref{isometry-defect}.
\end{remark}

\bigskip


\subsection{Dimension reduction}

The next step is to reduce the problem from three to two dimensions through slicing. Before we give the main result, we state  some classical slicing results in $\R^n$:

We fix $m>0$ and assume that every function in this subsection takes values in $[-m,m]$.
Let $A$ be a bounded open subset of $\mathbb R^n$, $e$ is a unit vector in $\mathbb R^n$ and $u$ a function on $A$. We denote by $M$ the orthogonal complement of $e$, and by $A_e$ the projection of $A$ onto $M$.  For every $z\in M$, we set $A_e^z:=\{t\in \mathbb R :  z+te\in A\}$; and $u_e^z$ to be the trace of $u$ on $A^z_e$, that is $u_e^z:=u(z+te)$.

\begin{prop}[section 5.10, page 216, in \cite{Evans-Gariepy}]
\label{prop-Evans-Gariepy}
Let $B\subset A$ be a given Borel set. If $B$ has finite perimeter in $A$, then $B_e^z$ has finite perimeter in $A_e^z$ and $\partial(B_e^z\cap A_e^z)=(\partial B\cap A)_e^z$ for a.e. $z\in A_e$, and
\be\label{formula-Evans-Gariepy}
\int_{A_e} \mathcal H^0(\partial B_e^z \cap A^z_e)dz=\int_{\partial B\cap A} \langle v_B,e\rangle.\ee
Conversely, $B$ has finite perimeter in $A$ if there exist $n$ linearly independent unit vectors $e$ such that the integral of $\mathcal H^0(\partial B_e^z \cap A_e^z)$ over all $z\in A_e$ is finite.
\end{prop}

From here we can establish a connection between the compactness of a family of functions in $L^1(\mathbb R^3)$ and the compactness of the traces of these functions. For every family $\mathcal F$ of functions of $A$, we set $\mathcal F_e^z:=\{u_e^z : u\in\mathcal F\}$, so that $\mathcal F_e^z$ is a family of functions on $A_e^z$. We say that a family $\mathcal F'$ is $\delta$-dense in $\mathcal F$ if $\mathcal F$ lies in a $\delta$-neighborhood of $\mathcal F'$ with respect to to the $L^1(A)$ topology. Then

\begin{thm}[theorem 6.6 in \cite{Alberti-Bouchitte-Seppecher:Phase-transition}]\label{thm-family-compactness}
Let $\mathcal F$ be a family of functions $v:A\to [-m,m]$ and assume that there exists $n$ linearly independent unit vectors $e$ which satisfy the following property:
\begin{myindentpar}{1cm}
For every $\delta>0$ there exists a family $\mathcal F_\delta$ $\delta$-dense in $\mathcal F$ such that $(\mathcal F_\delta)_e^z$ is pre-compact in $L^1(A_e^z)$ for $\mathcal H^{n-1}$ a.e. $z\in A_e$.
\end{myindentpar}
Then $\mathcal F$ is precompact in $L^1(A)$.
\end{thm}

Now we can give the main proposition of the subsection, where we slice a cylinder $D_r=B_r\times (0,r)\subset \R^3_+$.
We drop the subindex $r$ in the notation. We set coordinates in $D$ as $(x_1,x_2,y)$ where $x_1,x_{2}$ parameterize $B_r$ and $y\in(0,r)$. Then $E=\{y=0\}\cap \bar D$. Fix an arbitrary unit vector $e$ in the plane $\{y=0\}$, and let $E_e$ the projection of the set $E$.
We will slice $D$ in the direction of $y$, perpendicularly to $e$. The slice corresponding for each $z\in E_e$ is denoted by $D^z$; let $E^z$ its projection onto the plane $\{y=0\}$. For a nice picture, see figure 4 in \cite{Alberti-Bouchitte-Seppecher:Phase-transition}.

\begin{prop} \label{prop-boundary-term}
 Let $(u_\ep)\subset W^{1,2}(D,w_a)$ be a sequence with uniformly bounded energies  $F^a_\ep[u_\ep,D,E]$. Then the traces of $u_\ep$ are pre-compact in $L^1(E)$ and every cluster point belongs to $BV(E,I')$. Moreover, if  $Tu_\ep\to v$ in $L^1(E)$, then
\be\label{slicing0}\liminf_{\ep \to 0} F^a_\ep[u_\ep,D,E]\geq \kappa_s \abs{\int_{E\cap S_v} \nu_v} d\mathcal H^1.\ee
where the constant $\kappa_s$ is defined in \eqref{kappa}.
\end{prop}

\begin{proof}
We slice the region $D$ using Fubini's theorem, so then
\be\label{fubini}\begin{split}
F^a_\epsilon[u_\ep,D,E] & \geq \ep^{1-a}\int_D \grad {u_\epsilon} 2 w_a+\lambda_\ep \int_E V(Tu) \\
& \geq \int_{E_e} \lp \ep^{1-a}\int_{D^z}\grad {u_\epsilon ^z} 2 y^a+\lambda_\ep\int_{E^z} V(Tu_\epsilon^z)\rp dz.
\end{split}\ee
Next, we use the trace inequalities of section \ref{section-traces} in each domain $D^z$, $E^z$ (note that the inequalities are invariant by rescaling). Indeed, because of \eqref{trace-inequality-domain}, we have the estimate
\be\label{slicing1}\begin{split}
F^a_\epsilon[u_\ep,D,E] & \geq \int_{E_e} \lp \ep^{1-a}D_s^{-1}\int_{E^z\times E^z}\frac{\abs{Tu^z(x)-Tu^z(x')}^2}{\abs{x-x'}^{1+2s}}+\lambda_\ep\int_{E^z} V(Tu^z)\rp dz\\
& =\int_{E_e} G^a_\ep[Tu^z,E^z]dz.
\end{split}\ee
This last functional $G^a_\epsilon$ has been well studied in section \ref{section-perturbation} when $E^z$ is an interval in $\R$.

The rest of the proof follows exactly as proposition 4.7 in \cite{Alberti-Bouchitte-Seppecher:Phase-transition}. For commodity of the reader, we give the main ideas. We check first that $(Tu_\ep)$ is pre-compact in $L^1(E)$. Thanks to theorem \ref{thm-family-compactness}, it suffices to show that the family $\mathcal F:=(Tu_\epsilon)$ satisfy the following property: for every $\delta>0$, there exists a family $\mathcal F_\delta$ $\delta$-dense in $\mathcal F$ such that $(\mathcal F_\delta)_e^z$ is pre-compact in $L^1(E)$ for $\mathcal H^2$-a.e. $z \in E_e$. By assumption $F^a_\epsilon[u_\epsilon,D,E]\leq C$, so that \eqref{slicing1} implies also that
\be\label{slicing2}\int_{E_e} G^a_\epsilon[u_\epsilon^z, E^z]\leq C.\ee
Fix $\delta>0$.  For every $\delta>0$, define $v_\epsilon:E\to [-m,m]$ such that
\be\label{slicing3}
v_\ep^z:=
\left\{
  \begin{array}{ll}
    Tu_\epsilon^z, & \hbox{if }z\in E_e \mbox{ and }G^a_\ep[u_\epsilon^z, E^z]\leq 2mrC/\delta \\
    \alpha', & \hbox{otherwise.}
  \end{array}
\right.
\ee
By \eqref{slicing1}, \eqref{slicing2}, \eqref{slicing3}, we have $v_\epsilon^z=Tu_\epsilon^z$ for all $z\in E_e$ apart from a subset of measure smaller than $\delta/(2mr)$. Hence $v_\epsilon=Tu_\epsilon$ in $E$ up to a set of measure smaller than $\delta/m$. So, from $\abs{Tu_\epsilon}\leq m$, we deduce that $\norm{v_\epsilon-Tu_\epsilon}_{L^1(E)}\leq\delta$.
Therefore, the family $\mathcal F_\delta:=\{v_\epsilon\}$ is $\delta$-dense in $\mathcal F$.

Next, by \eqref{slicing3}, we have that $G^a_\ep[v_\ep^z,E^z]\leq 2mrC/\delta$ for every $z\in E_e$ and every $\ep$. Hence, theorem \ref{theorem-convergence1} implies that the sequence $(v_\epsilon^z)$ is pre-compact in $L^1(E^z)$. Thus $\mathcal F$ satisfies the hypothesis of theorem \ref{thm-family-compactness} for every $e$, and the sequence $(Tu_\epsilon)$ is pre-compact in $L^1(E)$.

It remains to prove that if $Tu_\epsilon \to E$ in $L^1(E)$, then $v$ belongs to $BV(E,I')$, and inequality  \eqref{slicing0} holds. We have (up to a subsequence), that $Tu_\epsilon^z \to v^z$ in $L^1(E)$ for a.e. $z\in E_e$ (remark 6.7 in \cite{Alberti-Bouchitte-Seppecher:Phase-transition}). Then, proposition \ref{theorem-convergence1} yields $v^z\in BV(E^z,I')$ and
\be\label{slicing10}
\liminf_{\ep\to 0} F^a_\epsilon[u_\ep,D,E]\geq \int_{E_e} \kappa_s \HH^0(Sv^z)dz.
\ee
The right hand side of the formula above is finite, so then proposition \ref{prop-Evans-Gariepy} implies that $v$ belongs to $BV(E,I')$, and that $S_{v^z}$ agrees with $S_v\cap E^z$ for a.e. $z\in E_\ep$. Then, by \eqref{formula-Evans-Gariepy} we may rewrite \eqref{slicing10} as
\bee
\liminf_{\ep\to 0} F^a_\epsilon[u_\ep,D,E]\geq \kappa_s \int_{E\cap S_v} \langle \nu_v, e\rangle dz.
\eee
Finally, \eqref{slicing0} follows by choosing a suitable unit vector $e$ in the expression above.

\end{proof}

\bigskip


\subsection{The boundary effect}\label{subsection-boundary-effect}

In the previous arguments we have reduced the dimension from three to two, so it is enough to understand the following functional on $\mathbb R^2_+$. Set $D=(-1,1)\times (0,1)$, $E=(-1,1)$, we define
$$H^a_\ep[w, D, E]:=\ep^{1-a}\int_{D} \grad {w} 2 y^a \;dxdy+\lambda_\ep\int_E V(Tw).$$
Trace inequalities allow to quickly relate this functional to the optimal profile obtained in proposition \ref{prop-optimum-profile}. This link is precisely the missing ingredient in \cite{Palatucci:Gamma-convergence} for the superquadratic case.

Let $\phi:\mathbb R\to [\alpha',\beta']$ be such optimal profile. It achieves the infimum of
$$\kappa_s:=\inf \left\{ G_1^a[v,\mathbb R] :
v\in H^s(\mathbb R),\; \lim_{x\to -\infty} v(x)=\alpha',\; \lim_{x\to +\infty} v(x)=\beta'\right\}.$$
Now let $w_1:=P *_x \phi$. We rescale
\be\label{function-boundary}
w_\ep(x,y):=w_1\lp \frac{x}{\Lambda_\ep},\frac{y}{\Lambda_\ep}\rp,\ee
and
$$\phi_\ep(x):=\phi\lp\frac{x}{\Lambda\ep}\rp,$$
Note that $w_\ep=P *_x \phi_\ep$.

\begin{lemma}\label{lemma-boundary-effect}
In the hypothesis above, we have that
$$H^a_\ep[w_\ep,D,E]=\kappa_s+o(1), \quad \mbox{when }\ep \to 0.$$
\end{lemma}

\begin{proof}
We note that, because of our rescaling,
$$H^a_\ep[w_\ep,D,E]=H_1^a[w_1,D/\Lambda_\ep, E/\Lambda_\ep].$$
But we can compute explicitly that
$$H^a_1[w_\ep,D/\Lambda_\ep,E/\Lambda_\ep]=H^a_1[w_1,\mathbb R^2_+,\mathbb R]-o(1).$$
On the other hand, because of the definition of $w_1$, we have equality in corollary  \ref{cor-sharp-halfplane}, and thus
$$H^a_1[w_1,\mathbb R^2_+,\mathbb R]=\frac{1}{D_s}\int_{\R^2}\frac{\abs{\phi(x)-\phi(x')}^2}{\abs{x-x'}^{1+2s}}\;dxdx'+\int_{\mathbb R} V(\phi)\;dx=G^a_1[\phi,\mathbb R]=\kappa_s.$$
The lemma is proved.
\end{proof}


\subsection{The wall effect}

Here we deal with the second term in the limit functional \eqref{functional-limit-Phi}:

\begin{prop}\label{prop-wall}
Let $A\subset \Omega\subset\R^3$ be a domain with piecewise $\mathcal C^1$ boundary, and $A'=\partial A\cap \partial\Omega$ with Lipschitz boundary. Let $u\in BV(A,I)$, $v\in BV(A',I')$ be given. Then
\begin{itemize}
\item [i.]For every sequence $(u_\ep)\subset W^{1,2}(A,w_a)$ such that $u_\ep\to u$ in $L^1(A)$ and $Tu_\ep\to v$ in $L^1(A')$,
$$\liminf_{\ep\to 0} E^a_\ep[u_\ep,A]\geq \int_{A'} \abs{\mathcal W(Tu)-\mathcal W(v)}.$$
\item[ii.] If $v$ is constant on $A'$ and $u$ is constant in $A$ with $u=\alpha$ or $u=\beta$, there exists a sequence $(u_\ep)$ such that $Tu_\ep=v$ in $A'$, $(u_\ep)$ converges to $u$, uniformly on every set with positive distance from $A'$ and
$$\limsup_{\ep\to 0} E^a_\ep[u_\ep,A]\leq \int_{A'} \abs{\mathcal W(Tu)-\mathcal W(v)}.$$
Moreover, the function $u_\epsilon$ may be required to be $\frac{C}{\epsilon}\lp \frac{r}{\ep}\rp^{-a}$-Lipschitz continuous in
$$A_r:=\left\{x\in A\mbox{ : } \dist(x,\partial A)\leq r\right\}.$$
\end{itemize}
\end{prop}

\begin{proof}
The proof is a modification of propositions 1.2 and 1.4  in \cite{Modica:phase-transitions-2}, and it is very well written for the case $p>2$ in Palatucci's PhD thesis (\cite{Palatucci:thesis}, proposition 4.3). Here we indicate the steps required, and only give the proof for the ones that require any modification.

We may assume that $E^a_\ep[u_\ep,A]\leq C$. For every $\ep>0$, let us denote
\be\label{wall1}
w_\ep(x):=(\mathcal W \circ u_\ep)(x),\quad\mbox{for }x\in A.\ee

\emph{Step 1: } We claim that $\int_{A}\abs{Dw_\ep}\leq \mbox{constant}$. Indeed, by Young's inequality we have
$$\int_A \abs{Dw_\ep}=\int_A \abs{\mathcal W'(u_\ep)}\abs{Du_\ep}=2\int_{A}\sqrt{W(u_\ep)}\abs{Du_\ep}\leq E^a_\ep[u_\ep,A]\leq C.$$

\emph{Step 2: } $w_\ep \to \mathcal W\circ u\in BV(A)$ in $L^1(A)$.\\

\emph{Step 3: } The functional
$$G_0(z):=\int_A \abs{Dz(x)}+\int_{\partial A} \abs{Tz-\mathcal W(v)}d\HH^2$$
is l.s.c. on $BV(A)$ with respect to the topology in $L^1(A)$.\\

\emph{Step 4: } Proof of statement \emph{i}. Applying lower semi-continuity of the functional $G_0$ to the sequence $w_\ep$ defined by \eqref{wall1} we obtain the following inequality
\be\label{wall2}
\begin{split}
\int_A \abs{D(\mathcal W \circ u)(x)} &+\int_{\partial A}\abs{\mathcal W(Tu)-W(v)}d\HH^2\\
& \leq \liminf_{\ep\to 0}\lp \int_A \abs{D(\mathcal W \circ u_\ep)(x)}+\int_{\partial A}\abs{\mathcal W(Tu_\ep)-W(v)}d\HH^2\rp\\
& \leq \liminf_{\ep\to 0}\lp E^a_\ep[u_\ep,A]+\int_{\partial A}\abs{\mathcal W(Tu_\ep)-W(v)}d\HH^2\rp,
\end{split}\ee
by Young's inequality. Next, since $Tu_\ep \to v$ in $L^1(A')$, we deduce that
\be\label{wall3}\liminf_{\ep \to 0} \int_{\partial A} \abs{\mathcal W(Tu_\ep)-\mathcal W(v)}d\HH^2=0.\ee
From \eqref{wall2} and \eqref{wall3} we obtain the lower bound inequality \emph{i}.\\

\emph{Step 5:} Proof of the upper bound \emph{ii}. The weight $w_a$ needs to be taken into account. Without loss of generality, consider the case $u=\beta$ and $v=\gamma$ with $\alpha<\gamma<\beta$; the other cases are similar. Let $\theta:[0,+\infty)\to[\gamma,\beta]$ be the solution of the ODE written as
\be\label{ODE-theta}\left\{\begin{split}
\theta'&=\sqrt{W(\theta)} \\
\theta (0)&=\gamma.
\end{split}\right.
\ee
Let $d(x)=\dist(x,A')=\dist(x,\partial\Omega)$. We set
$\phi(t):=\theta(\omega)$ for $\omega:=\frac{t^{1-a}}{1-a}$, and $u_\epsilon(x):=\phi\lp \frac{d(x)}{\epsilon}\rp$. Then
$$E^a_\epsilon[u_\epsilon,A]=\epsilon^{1-a}\int_A \grad {u_\epsilon} 2h^a+\frac{1}{\epsilon^{1-a}}\int_{A}W(u_\epsilon)h^{-a},$$
that can be written by the coarea formula as
\bee
E^a_\epsilon[u_\epsilon,A] =\int_{\mathbb R^+}\int_{\Sigma_{\epsilon t}}\left[ \phi'(t)^ 2 t^a+ W(\phi(t))t^{-a}\right]d\sigma dt,
\eee
and after the change $\omega=\frac{t^{1-a}}{1-a}$ we get
\bee
E^a_\epsilon[u_\epsilon,A] =\int_{\mathbb R^+}\int_{\Sigma_{\epsilon ((1-a)\omega)^{\frac{1}{1-a}}}}
\left[ \theta'(\omega)^2 +W(\theta(\omega))\right] d\sigma d\omega,
\eee
where
$\Sigma_{s}$ is the set of points in $A$ at a distance exactly $s$ from $A'$. When $\epsilon\to 0$, we know that $\Sigma_{O(\ep)}\to A'$, and thus we have that
\bee
E^a_\epsilon[u_\epsilon,A]\to \int_{\mathbb R^+}\int_{A'} \left[ \theta'(\omega)^2 +W(\theta(\omega))\right] d\sigma d\omega.
\eee
But because $\theta$ satisfies the ODE \eqref{ODE-theta}, then both terms in the above expression are equal, so the inequality $x_1^2+x_2^2\geq 2x_1 x_2$ becomes an equality and we can conclude that
$$E^a_\epsilon[u_\epsilon,A]\to \int_{A'}\int_{\mathbb R^+} 2\sqrt{W(\theta)}\theta'd\omega d\sigma=\int_{A'}\left[\mathcal W(\beta)-\mathcal W(\gamma)\right]d\sigma,$$
where we take  $\mathcal W$ to be a primitive of $2\sqrt W$. To finish the proof of the proposition, just note that
$$\sup\abs{\nabla u_\ep}\leq \frac{Cr^{-a}}{\ep^{1-a}}.$$
\end{proof}


\section{Proof of theorem \ref{main-theorem}}

\setcounter{equation}{00}

Once we have the main ingredients from the previous sections, we can give the proof of the main theorem.

\subsection{Compactness}

Let $(u_\epsilon)$ be a sequence in $W^{1,2}(\Omega,w_a)$ such that $F^a_\ep[u_\ep]$ is bounded. Using the localization defined \eqref{localization-functional} and the functional $E^a_\ep$ from \eqref{functional-interior} we know that
$$F^a_\epsilon[u_\epsilon]\geq F^a_\epsilon[u_\epsilon,\Omega,\emptyset]=E^a_\epsilon[u_\epsilon,\Omega].$$
By statement \emph{iii.} of proposition \ref{prop-interior} we conclude that $(u_\ep)$ is precompact in $L^1(\Omega)$ and there exists $u\in BV(\Omega,I)$ such that $u_\epsilon\to u$ in $L^1(\Omega)$.

It remains to prove that $(Tu_\epsilon)$ is pre-compact in $L^1(\partial\Omega)$ and every cluster point belongs to $BV(\partial\Omega,I')$. Thanks to proposition \ref{prop-flatten}, we can cover $\partial\Omega$ with finitely many ``cubes" $(Q_j)_{j\in J}$, centered on $\partial\Omega$, of radius $r_j$, such that for every $j\in J$, there exists a bi-Lipschitz map $\Psi_j$ with isometry defect $\delta(\Psi_i)<1$, which satisfies $\Psi_j(D_{r_j}\cap Q_j)=\Omega\cap Q_j$ and $\Psi_j(E_{r_j}\cap Q_j)=\partial\Omega \cap B_j$.

We show that $(T u_\ep)$ is pre-compact in $L^1(\partial\Omega\cap Q_j)$ for every $j\in J$. For every fixed $j$, let us set $u_\ep^j:=u_\ep\circ \Psi_j$. We have that
$$F_\ep^a[u_\ep,\Omega\cap Q_j,\partial \Omega\cap Q_j]\geq \lp 1-\delta(\Psi_j)\rp^5 F^a_\ep[u_\ep^j, D_{r_j}\cap B_j,E_{r_j}\cap B_j],$$
so we conclude that $F_\ep^a[u_\ep^j, D_{r_j}\cap B_j,E_{r_j}\cap B_j]$ is uniformly bounded in epsilon. Hence, the compactness of the traces $Tu_\ep^i$ in $L^1(E_{r_j})$ follows from proposition \ref{prop-boundary-term}. Finally, using the invertibility of $\Psi_j$, we have that $(Tu_\ep)$ is pre-compact in $L^1(\partial\Omega)$ and that its cluster points are in $BV(\partial\Omega, I')$.


\subsection{Lower bound inequality}

Now we continue with the proof of the theorem. Parts  \emph{i.} and \emph{iii.} follow by putting together the results in the interior (section \ref{section-interior}) and the boundary (sections \ref{section-perturbation} and \ref{section-preliminary}).

Let $(u_\ep)$ be a sequence in $W^{1,2}(\Omega,w_a)$ satisfying $u\in BV(\Omega,I)$, $v\in BV(\partial\Omega,I')$, and such that $u_\ep \to u$ in $L^1(\Omega)$, $Tu_\ep \to v$ in $L^1(\partial \Omega)$. We have to show that
\be\label{lower1}\liminf_{\ep \to 0} F^a_\ep[u_\ep]\geq \Phi(u,v),\ee
where $\Phi$ is given by \eqref{functional-limit-Phi}. Assume, without loss of generality, that $\liminf_{\ep\to 0} F^a_\ep[u_\ep]<+\infty.$

For every $\ep>0$, let $\mu_\ep$ be the energy distribution associated with $F^a_\ep$ with configuration $u_\ep$, i.e., $\mu_\ep$ is the positive measure given by
$$\mu_\ep(B):=\epsilon^{1-a}\int_{\Omega\cap B}\grad u 2 h^a+\tfrac{1}{\ep^{1-a}}\int_{\Omega\cap B} W(u)h^{-a}+\lambda_\ep\int_{\partial\Omega\cap B} V(Tu),$$
for every $B\subset \R^3$. Similarly, we define
\bee\begin{split}
& \mu^1(B):= \sigma \HH^2(S_u\cap B), \\
& \mu^2(B):= \int_{\partial\Omega \cap B} \abs{\mathcal W(Tu)-\mathcal W(v)}d\HH^2,\\
& \mu^3(B):=\kappa_s \HH^1(S_v\cap B).
\end{split}\eee
The total variation $\norm{\mu_\ep}$ of the measure $\mu_\ep$ is equal to $F^a_\ep[u_\ep]$, and $\norm{\mu^1}+\norm{\mu^2}+\norm{\mu^3}$ is equal to $\Phi(u,v)$. Note that $\norm{\mu_\ep}$ is bounded, so we can assume that $\mu_\ep$ converges in the sense of measure to some finite measure $\mu$. Then, by the lower semicontinuity of the total variation we have
$$\liminf_{\ep \to 0} F^a_\ep[u_\ep]=\liminf_{\ep \to 0} \norm{\mu_\ep}\geq \norm{\mu}.$$
Since the measures $\mu^i$ are mutually singular, we obtain the lower bound inequality \eqref{lower1} if we prove that
$$\mu\geq \mu^i \quad\mbox{for }i=1,2,3.$$
It is enough to show that $\mu(B)\geq \mu^i(B)$ for all sets $B\subset \mathbb R^3$ such that $B\cap \Omega$ is a Lipschitz domain and $\mu(\partial B)=0$.

First, because of proposition \ref{prop-interior} we have that
$$\mu(B)=\lim_{\ep\to 0} \mu_\ep(B)\geq \liminf_{\ep\to 0} F^a_\ep[u_\ep,\Omega\cap B, \emptyset]\geq
\sigma \HH^2(S_u\cap B)=\mu^1(B).$$
Similarly, we can prove that $\mu\geq \mu^2$. More precisely,
$$\mu(B)=\lim_{\ep \to 0} \mu_\ep(B)\geq \liminf_{\ep\to 0} F_\ep^a[u_\ep,\Omega\cap B,\emptyset]
\geq  \int_{\partial\Omega\cap B} \abs{\mathcal W(T(u))-\mathcal W(v)}d\HH^2=\mu^2(B)$$
where we have used proposition \ref{prop-wall} with $A=B\cap \Omega$ and $A'=B\cap \partial\Omega$.\\

The inequality $\mu\geq \mu^3$ requires a different argument. Notice that $\mu^3$ is the restriction of $\HH^1$ to the set $S_v$, multiplied by the factor $\kappa_s$. Thus, if we prove that
\be\liminf_{r\to 0} \frac{\mu(Q_r(x))}{2r}\geq \kappa_s, \quad \HH^1 \mbox{a.e. }x\in S_v,\ee
for $Q_r(x)$ as in proposition \ref{prop-flatten}, we obtain the required inequality. Note that, in any case, $\mu$ is supported on $\bar\Omega$.

Let us fix $x\in S_v$ such that there exists $\lim_{r \to 0}\frac{\mu(Q_r(x))}{2r}$, and $S_v$ has one-dimensional density equal to $1$. We denote by $\nu_v$ the unit normal at $x$. For $r$ small enough, we choose a map $\Psi_r$ as in proposition \ref{prop-flatten}. Set $\bar u_\ep:= u_\ep\circ \Psi_r$ and $\bar v:= v\circ \Psi_r$. Hence, $T\bar u_\ep \to \bar v$ in $L^1(E_r)$ and $\bar v\in BV(E_r,I')$. Moreover,
\be\begin{split}
\mu(Q_r(x)) &
=\lim_{\ep \to 0} \mu_\ep(Q_r(x))\\
& =\lim_{\ep \to 0} F_\ep^a[u_\ep,\Omega\cap Q_r(x),\partial\Omega\cap Q_r(x)]\\
& \geq \liminf_{\ep \to 0} (1-\delta(\Psi_r))^{5} F^a_\ep[\bar u_\ep,D_r, E_r].
\end{split}\ee
On the other hand, by proposition \ref{prop-boundary-term}, we have that
$$\liminf_{\ep\to 0} F^a_\ep[\bar u_\ep, D_r, E_r]\geq \kappa_s\abs{\int_{S_{\bar v}\cap E_r} \nu_v d\HH^1}.$$
Finally, notice that $\delta(\Psi_r)\to 0$ as $r\to 0$, and that
$$\abs{\int_{S_{\bar v}\cap E_r} \nu_v d\HH^1}=2r+o(r).$$
Thus we obtain that
$$\frac{\mu(Q_r(x))}{2r}\geq \kappa_s \lp 1+\frac{o(r)}{2r}\rp, \quad \mbox{as } r \to 0,$$
that implies $\mu\geq \mu^3$. The proof of the lower bound inequality is completed.

\subsection{Upper bound inequality}

For the proof of \emph{ii.}, we use a standard construction piece by piece. We will require an extension lemma

\begin{lemma}\label{lemma-extension}
Let $A$ be a domain in $\mathbb R^3$, that is contained in the strip $\{r<\dist(\cdot,\partial\Omega)<2r\}$, and let $A'\subset \partial A$, $v:A'\to [-m,m]$ a Lipschitz function.
Then, for every $\epsilon>0$, there exists an extension $u:\bar A\to [-m,m]$ such that
$$\mbox{Lip}(u)\leq \frac{1}{\epsilon^{1-a}r^a}+\mbox{Lip} (v)$$
and
$$E^a_\epsilon[u,A]\lesssim \left[\lp\epsilon^{1-a}r^{a} \mbox{Lip}(v)+1\rp^2+C_m\right]\lp \mathcal H^2(\partial A)+o(1)\rp z,\quad \mbox{as }\epsilon\to 0.$$
where
$$C_m:=\max_{t\in[-m,m]} W(t),\quad z=\min\{\norm{v-\alpha}_{L^\infty}, \norm{v-\beta}_{L^\infty}\}.$$
\end{lemma}

\begin{proof}
It follows the ideas of lemma 4.11 in \cite{Alberti-Bouchitte-Seppecher:Phase-transition}, but we need to take care of the weight $h$. First, we assume without loss of generality, that $A'=\partial A$. In fact, we can extend $v$ to $\partial A$ without increasing its Lipschitz constant. We additionally suppose that $z=\norm{v-\alpha}_{\infty}$; the other case $z=\norm{v-\beta}_{\infty}$ is similar.

Let $\delta=\ep^{1-a}r^a$, and set
\bee u(x):=
\left\{
  \begin{array}{ll}
    v(x), & \hbox{on }\partial A,  \\
    \alpha, & \hbox{on } A\backslash A_{z\delta},
  \end{array}
\right.
\eee
where $A_t$ is the set of all $x$ in $A$ such that $0<\dist(x,\partial A)<t$. Then, $u$ is $\lp \frac{1}{\delta}+\mbox{Lip} (v)\rp$-Lipschitz continuous on $\bar A\backslash A_{z\delta}$. Finally, $u$ can be extended to $\bar A$, without increasing its Lipschitz constant. We have
\bee\begin{split}
E^a_\ep[u,A] &
=\ep^{1-a}\int_{A_{z\delta}} \grad u  2 h^a +\frac{1}{\ep^{1-a}}\int_{A_{z\delta}} W(u)h^{-a} \\
& \lesssim \abs{A_{z\delta}}\left[ \ep^{1-a} \lp\frac{1}{\delta}+\mbox{Lip} (v)\rp^2 r^a+\frac{1}{\ep^{1-a}}C_m r^{-a}\right]\\
& \leq \left[ \lp \HH^2(\partial A)+o(1)\rp \lp \delta \mbox{Lip}(v)+1\rp^2 +C_m\right]z, \quad \mbox{as }\ep \to 0,
\end{split}\eee
where we have used that $\abs{A_t}= t\lp\HH^2(\partial A)+o(1)\rp$ as $t\to 0$.
\end{proof}

\bigskip

Now we are ready for the proof of the upper bound in theorem \ref{main-theorem}. Fix $u\in BV(\Omega,I)$ and $v\in BV(\partial\Omega,I')$. It is enough to assume that the singular sets of $u$ and $v$, $S_u$ and $S_v$ respectively, are closed manifolds of class $\mathcal C^2$ without boundary. This is so because every pair $(u,v)\in BV(\Omega,I)\times BV(\partial\Omega,I')$ can be approximated in $\in L^1(\Omega)\times L^1(\partial\Omega)$ by pairs that fulfils those regularity assumptions (see theorem  1.24 of \cite{Giusti}).  We assume that $u$ and $v$, up to modifications on negligible sets, are constant in each connected component of $\Omega\back S_u$ and $\partial\Omega\backslash S_v$ respectively. \\

The idea is to construct a partition of $\Omega$ into four subsets, and to use the preliminary convergence results of the previous sections to obtain the upper bound inequality.

For every $x\in \Omega$, set $d(x)=\dist(x,\partial\Omega)$ and $d':\partial\Omega\to \R$ is the oriented distance from $S_v$ defined by
\bee d'(x)=\left\{
  \begin{array}{ll}
    \dist(x,S_v)&\mbox{ if } x\in \{v=\beta'\}, \\
    -\dist(x,S_v)&\mbox{ if } x\in\{v=\alpha'\}.
  \end{array}
\right.
\eee

For every $r>0$, set
$$\Gamma_r:=\{ x\in \Omega : \dist(x,\partial\Omega)=r\}.$$
Fix $r>0$ such that $\Gamma_r$ and $\Gamma_{2r}$ are Lipschitz surfaces and $S_u\cap \Gamma_r$ is a Lipschitz curve.
With this is mind, we construct a partition of $\Omega$. Let
\bee\begin{split}
&\quad B_1:=\{x\in\Omega : \dist(x,S_v\cup (S_u\cap \Gamma_r))<3r\},\\
&\quad A_1:=\{x\in\Omega\backslash \bar B_1 : d(x)<r\}, \\
&\quad B_2:=\{x\in\Omega\backslash\bar B_1 : r<d(x)<2r\}, \\
&\quad A_2:=\{x\in\Omega\backslash \bar B_1 : 2r<d(x)\}.
\end{split}\eee
We will construct a Lipschitz function $u_\ep:=u_{r,\epsilon}$ for every $\ep<r$, piece by piece, with controlled Lipschitz constant. \\

\emph{Step 1: }
In the set $A_2$, we take $u_\ep$ as in part \emph{ii.} of proposition \ref{prop-interior}. We extend it to $\partial A_2$ by continuity. Hence $u_\ep$ is $\frac{C}{\ep^{1-a} r^a}$-Lipschitz on $\bar A_2$, $u_\ep$ converges pointwise to $u$ in $A_2$, uniformly on $\partial A_2\cap \partial B_2$, and
$$F^a_\ep[u_\ep, A_2,\emptyset]=E^a_\ep[u_\ep, A_2]\leq\sigma \HH^2(S_u\cap A_2)+o(1)\leq \sigma\HH^2(S_u)-\sigma\HH^2 \lp S_u\backslash A_2\rp +o(1)$$
as $\ep \to 0$. \\

\emph{Step 2: } Now we consider the set $A_1$.
The function $u$ is constant (equal to $\alpha$ or $\beta$) in every connected component $A$ of $A_1$ on $\partial A\cap \partial\Omega$, and the function $v$ is constant (equal to $\alpha'$ or $\beta'$) on $\partial A\cap\partial\Omega$. Then we can use proposition \ref{prop-wall} to get a function $u_\ep$ such that $Tu_\ep=v$ on $\partial A\cap \partial\Omega$ and $u_\ep$ converges to $u$ pointwise in $A_1$ and uniformly on every subset with positive distance from $\partial A\cap\partial\Omega$. By the same proposition, we also have that $u_\ep$ is $\frac{C_W}{\ep^{1-a}r^a}$-Lipschitz continuous on $\bar A_1$, and we can extend it to $\partial A_1$ by continuity. Since the distance of two connected components of $A_1$ is larger than $r$ and $\frac{1}{\ep^{1-a}r^a}>\frac{1}{r}$, choosing $C>\max\{2m,C_W\}$, it follows that $u_\ep$ is $\frac{C}{\ep^{1-a}r^a}$-Lipschitz continuous on $\bar A_1$ and agrees with $v$ on $\partial A_1\cap \partial\Omega$. Moreover, the function $u_\ep$ satisfies
$$F^a_\ep[u_\ep,A_1,\partial A_1\cap\partial\Omega]=E^a_\epsilon[u_\ep,A_1]\leq \int_{\partial A_1\cap \partial\Omega}\abs{\mathcal W(Tu_\ep)-\mathcal W(v)}+o(1), \quad \mbox{as }\ep \to 0.$$

\emph{Step 3: }
Note that in the previous steps we have constructed an optimal sequence in $\bar A_1 \cup \bar A_2$ that is $\frac{C}{\ep^{1-a}r^a}$-Lipschitz continuous, in particular, it is defined and Lipschitz on $((\partial A_1 \cup \partial A_2)\cap \partial B)$, for every connected component $B$ of $B_2$.

By virtue of lemma \ref{lemma-extension}, we can extend $u_\ep$ to every $B$, obtaining a $\frac{C+1}{\ep^{1-a}r^a}$-Lipschitz continuous function that satisfies
$$ F_\ep^a[u_\ep,B_2,\emptyset]=E^a_\ep[u_\ep,B_2]\leq z_\ep \lp (C+2)^2+C_m\rp\lp\HH^2(\partial B_2)+o(1)\rp=o(1)$$
as $\ep \to 0$, where we have used that $z_\ep:=\inf_{(\partial A_1 \cup \partial A_2)\cap \partial B_2} \abs{u_\ep -u}=o(1)$, since $u_\ep$ is constant on each connected components of $B_2$.\\

\emph{Step 4:}
To construct the function in the piece $B_1$ is the most delicate step. First, we need some preliminaries: construct a function on the whole $\R^2_+$ with suitable behavior.

Consider the rescalings
$$\Lambda_\ep<<\sigma_\ep<<\rho_\ep<<\ep^{1-a}$$
for some $\sigma_\ep=\ep^p$, $\rho_\ep=\ep^q$.  Let $\bar w_1$ be the function on  $\R^2_+$ defined in proposition \ref{prop-sharp} and its rescaling $\bar w_\ep(x,y):=\bar w_1\lp \frac{x}{\Lambda_\ep},\frac{y}{\Lambda_\ep}\rp$. We also consider the function $w_1$ from lemma \ref{lemma-boundary-effect} and its rescaling $w_\ep(x,y):= w_1\lp \frac{x}{\Lambda_\ep},\frac{y}{\Lambda_\ep}\rp$. We glue them, so that we obtain a function defined in the whole $\R^2_+$, as
\bee\tilde w_1 :=\left\{
                   \begin{array}{ll}
                     w_1, & \hbox{if } (x,y)\in D_{\sigma_\ep} \\
                     \bar w & \hbox{if } (x,y)\in \mathbb R^2_+ \backslash D_{\rho_\ep}
                   \end{array}
                 \right.
\eee
and smooth in between, with its corresponding rescaling
\be\label{tilde-w-epsilon}
\tilde w_\ep(x,y):=\tilde w_1\lp\frac{x}{\Lambda_\ep},\frac{y}{\Lambda_\ep} \rp.\ee
Because $\rho_\ep>>\sigma_\ep>>\Lambda_\ep$, we can apply lemma \ref{lemma-boundary-effect} to obtain
$$H^a_\ep[\tilde w_\ep,D_{\rho_\ep},E_{\rho_\ep}]=\kappa_s-o(1)$$
when $\ep \to 0$.

Now we pass from two to three dimensions.
In particular, we set $\tilde u_\ep$ on $S_v\times\mathbb R^2_+$ to be
\be\label{tilde-u}\tilde  u_\ep(x,y,z):=\tilde w_\ep(x,y)\quad\mbox{for every } z\in S_v, (x,y)\in\mathbb R^2_+.\ee
where $\tilde w_\ep$ is defined in \eqref{tilde-w-epsilon}. In addition, for any function $\tilde u$ defined on $S_v\times D$, we define the following functional.
\bee\begin{split}\tilde F_\ep^a[\tilde u, S_v\times D, S_v\times E] & :=  \ep^{1-a}\int_{S_v\times D} \grad {\tilde u} 2 y^a\;dxdydz \\
& +\frac{1}{\ep^{1-a}}\int_{S_v\times D}W(\tilde u)y^{-a}\;dxdydz+\lambda_\ep\int_{S_v\times E} V(T\tilde u)\;dxdz.
\end{split}\eee
For the $\tilde u_\ep$ we have constructed in \eqref{tilde-u}, Fubini's theorem implies that
\be\label{upper4}\begin{split}
\tilde F_\ep^a[\tilde u_\ep,S_v\times D_{\rho_\ep}, S_v\times E_{\rho_\ep}]
& = \HH^1(S_v) \lp H_\ep^a[\tilde w_\ep,D_{\rho_\ep}, E_{\rho_\ep}] + \frac{1}{\ep^{1-a}}\int_{D_{\rho_\ep}} W( \tilde w_\ep)y^{-a}\;dxdy\rp\\
& \leq \HH^1(S_v) \lp H_\ep^a[\tilde w_\ep,D_{\rho_\ep}, E_{\rho_\ep}]+\frac{C}{\ep^{1-a}}\int_{
x\in(-\rho_\ep,\rho_\ep), y\in(0,\rho_\ep)} y^{-a}\;dxdy\rp \\
& =\HH^1(S_v) \lp H_\ep^a[\tilde w_\ep,D_{\rho_\ep}, E_{\rho_\ep}]+ \frac{C'}{\ep^{1-a}} \rho_\ep^{2-a}\rp.
\end{split}\ee
We choose $1-a<p<q<\frac{1-a}{-a}$ and $q>\frac{1-a}{2-a}$. Then from \eqref{upper4} we obtain that
\be\label{eq100}\tilde F_\ep^a[\tilde u_\ep,S_v\times D_{\rho_\ep}, S_v\times E_{\rho_\ep}]\leq \HH^1(S_v) \left[ \kappa_s+o(1)\right]\ee
as $\ep\to 0$. \\

Now we transplant the function $\tilde u_\ep$ obtained to our remaining piece $B_1$.
Since $S_v$ is a boundary in $\partial\Omega$, we can construct a diffeomorphism between the intersection of a tubular neighborhood $U$ of $S_v$ and $\Omega$ and the product of $S_v$ with a half-disk. More precisely, for every $r>0$, we set
$$\mathcal S_r:=\{x\in\Omega : 0<\dist(x,S_v)<r\}.$$
For every $x\in\bar\Omega$, define
$$\Psi(x):=(x'',d'(x'),\dist(x,\partial\Omega)),$$
where $x'$ is a projection of $x$ on $\partial\Omega$ and $x''$ is a projection of $x'$ on $S_v$. The function $\Psi$ is well-defined and is a diffeomorphism of class $\mathcal C^2$ on $\bar \Omega\cap U$, and satisfies the following properties:
\begin{itemize}
\item $\Psi(\Omega\cap U)\subset S_v \times \mathbb R^2_+$,
\item $\Psi(\partial\Omega\cap U)\subset S_v\times\mathbb R\times \{0\}$,
\item $\Psi(x)=x$ for every $x\in\partial\Omega$.
\item $D\Psi(x)$ is an isometry,
\item $\lim_{r\to 0} \delta_r=0$, where $\delta_r$ is the isometry defect of the restriction of $\Psi$ to $\mathcal S_r$.
\end{itemize}
We construct $u_\ep$ on $S:=\overline {\mathcal S}_{\rho_\ep/2}$. With small modifications, we can assume that $S$ is a ``cubical'' neighborhood such that for $\ep$ small, the function $\Psi$ maps $S$ into $S_v\times D_{\rho_\ep}$ and $\partial S\cap\partial\Omega$ into $S_v\times E_{\rho_\ep}$,
Then we define $u_\ep:= \tilde u_\ep \circ \Psi$,
where $\tilde w_\ep$ is defined in \eqref{tilde-u}. Thus proposition \eqref{prop-flatten} and \eqref{eq100} give that
\bee\begin{split}
F^a_\ep[u_\ep,S, \partial S\cap \partial\Omega]
& \leq(1-\delta_\ep)^{-5}F^a_\ep[\tilde w_\ep, S_v\times D_{\rho_\ep}, S_v\times E_{\rho_\ep}] \\
& \leq \HH^{1}(S_v) \lp \kappa_s +o(1)\rp
\end{split}\eee
as $\ep\to 0$, because $\delta_\ep:=\delta(\Psi|_{\mathcal S_{\rho_\ep}})$ tends to zero as $\ep\to 0$.

Notice that for $\ep$ small enough, $\Psi$ is 2-Lipschitz continuous. Using again lemma \ref{lemma-extension}, we can extend $u_\ep$ by setting $u_\ep=v$ on the remaining part of $\partial B_1\cap \partial\Omega$. We have that $u_\ep$ is equal to $v$ on $\partial\Omega\backslash S$. Thus, we can extend $u_\ep$ on the whole $B_1\backslash S$ to a $\frac{2C+1}{\ep^{1-a}r^a}$-Lipschitz continuous function, which satisfies
\bee\begin{split}
F_\ep^a [u_\ep, B_1\backslash S,\partial (B_1\backslash \bar S)\cap \partial\Omega]
& = H^a_\ep[u_\ep, B_1\backslash \bar S] \\
& \leq \lp (2C+2)^2+C_m\rp \lp \HH^2(\partial B_1)+o(1)\rp 2m\end{split}\eee
as $\ep \to 0$, where we have used $\norm{u_\ep-\alpha}_{\infty}\wedge \norm{u_\ep-\beta}_{\infty}\leq 2m$.\\

\emph{Step 5: } We recall that for every $r>0$ and every $\ep<r$ we have constructed a function $u_\ep$ defined on the whole $\Omega$ such that
$$\limsup_{\ep \to 0}\norm{u_\ep-u}_{L^1(\Omega)}\leq 2m \lp \abs{B_1}+\abs{B_2}\rp$$
and
$$\limsup_{\ep \to 0} \norm {T u_\ep-v}_{L^1(\partial\Omega)}=0.$$
Since $\abs{B_1}$ and $\abs{B_2}$ have order $r^2$ and $r$ respectively, we get that $u_\ep \to u$ in $L^1(\Omega)$, first taking $\ep \to 0$ and then $r\to 0$.

Combining all the results above, we obtain
\bee\begin{split}
\limsup_{\ep \to 0} F_\ep^a[u_\ep]
&\leq \sigma\HH^2(S_u)+ \int_{\partial\Omega}\abs{\mathcal W(Tu(x))-\mathcal W(v(x))}d\HH^2 +\kappa_s\HH^1(S_v)\\
& -\sigma \HH^2(S_u\backslash A_2)+\lp (2C+2)^2+C_m\rp\lp\HH^2(\partial B_1)+o(1)\rp 2m.
\end{split}\eee
Since $\HH^2(\partial B_1)$ has order $r$, taking $r\to 0$ above we deduce the upper bound inequality. Finally, applying a suitable diagonalization argument, to the sequence $u_\ep:=u_{\ep,r}$ we obtain the desired sequence $u_\ep$. Proof of \emph{ii.} is completed.


\bigskip

\noindent {\bf Acknowledgements}\\
 I would like to thank Prof. Luis Caffarelli for suggesting the study of fractional order operators and his support at UT Austin.

\end{document}